\documentclass[a4paper,11pt]{article}
\usepackage[T1]{fontenc}
\usepackage{ae}
\usepackage{aecompl}
\usepackage{lmodern}
\usepackage[utf8]{inputenc}
\usepackage[english]{babel}
\usepackage[a4paper,scale=0.75]{geometry}
\usepackage[ruled]{algorithm2e} 
\usepackage{amsthm}
\usepackage{amsmath}
\usepackage{amsfonts}
\usepackage{amssymb}
\usepackage{mathrsfs}
\usepackage{url}
\usepackage[all]{xy}
\usepackage{makeidx}
\author{Enea Milio}
\title{Computing isogenies between Jacobians of curves of genus~2~and~3}
\date{}
\newcommand{\Z}{\mathbb{Z}}
\newcommand{\Q}{\mathbb{Q}}
\newcommand{\C}{\mathbb{C}}
\newcommand{\HH}{\mathcal{H}}

\newcommand{\bb}{\backslash}

\newenvironment{psmallmatrix}{\left [\begin{smallmatrix}}{\end{smallmatrix}\right ]}
\newenvironment{ppsmallmatrix}{\left (\begin{smallmatrix}}{\end{smallmatrix}\right )}
\newcommand{\thetacar}[2]{\theta\begin{psmallmatrix}#1\\ #2\end{psmallmatrix}}

\makeatletter
\newcommand*{\newaliascnt}[2]{%
  \begingroup
    \def\AC@glet##1{%
      \global\expandafter\let\csname##1#1\expandafter\endcsname
        \csname##1#2\endcsname
    }%
    \@ifundefined{c@#2}{%
      \@nocounterr{#2}%
    }{%
        \AC@glet{c@}%
        \AC@glet{the}%
        \AC@glet{theH}%
        \AC@glet{p@}%
        \expandafter\gdef\csname AC@cnt@#1\endcsname{#2}%
        \expandafter\gdef\csname cl@#1\expandafter\endcsname
        \expandafter{\csname cl@#2\endcsname}%
    }%
  \endgroup
}
\makeatother

\newtheorem{theo}{Theorem}
\newaliascnt{algocf}{theo}

\newtheorem{deff}[theo]{Definition}
\newtheorem{prop}[theo]{Proposition}
\newtheorem{coro}[theo]{Corollary}

\theoremstyle{definition}
\newtheorem{remm}[theo]{Remark}
\newtheorem{exem}[theo]{Example}

\newcommand{\lquo}[2]{\leavevmode\kern-.1em\lower.25ex\hbox{$#2$}\kern-.1em\backslash\kern-.1em\raise.2ex\hbox{$#1$}}
\newcommand{\rquo}[2]{\leavevmode\kern-.1em\raise.2ex\hbox{$#1$}\kern-.1em/\kern-.1em\lower.25ex\hbox{$#2$}}

\DeclareMathOperator{\tr}{tr}
\DeclareMathOperator{\Pic}{Pic}

\begin{document}

\maketitle

\begin{abstract}
  We present a quasi-linear algorithm to compute (separable) isogenies of
  degree $\ell^g$, for $\ell$ an odd prime number, between Jacobians of curves of genus $g=2$ and $3$ starting from the equation
  of the curve $\mathcal{C}$ and a maximal isotropic subgroup $\mathcal{V}$ of the $\ell$-torsion, generalizing Vélu's formula from genus $1$.
  Denoting by $J_{\mathcal{C}}$ the Jacobian of $\mathcal{C}$, the isogeny is $J_{\mathcal{C}}\to J_{\mathcal{C}}/\mathcal{V}$. Thus $\mathcal{V}$ is the kernel of the isogeny and we compute only isogenies with such kernels.
  This work is based on the paper \emph{Computing functions on Jacobians and their quotients} of Jean-Marc Couveignes and Tony Ezome.
  We improve their genus $2$  algorithm, generalize it to genus $3$ hyperelliptic curves
  and introduce a way to deal with the genus $3$ non-hyperelliptic case, using algebraic theta functions.
\end{abstract}

\section{Introduction}
  
Starting from a projective, smooth, absolutely integral curve $\mathcal{C}$ of genus $g\in\{2,3\}$ over a finite field $K$ and a maximal isotropic subgroup $\mathcal{V}$
of the $\ell$-torsion of the Jacobian $J_{\mathcal{C}}$ of $\mathcal{C}$, we want to compute the equation of a $(\ell,\ldots,\ell)$-isogenous
curve $\mathcal{D}$ such that $J_{\mathcal{D}}=J_{\mathcal{C}}/\mathcal{V}$ and equations for the isogeny $f:J_{\mathcal{C}}\to J_{\mathcal{D}}$ allowing one to compute the image of a point in $J_{\mathcal{C}}$ by the isogeny~$f$.
The computation of $\mathcal{V}$ is a different problem that we do not treat here. We take it as an input to our algorithms.
It is represented as a set of $\ell^g$ classes of divisors of the curve $\mathcal{C}$ over $\bar{K}$ in Mumford representation or
as a collection of fields extensions $(L_i/K)$ and points $w_i\in\mathcal{V}(L_i)$ in Mumford representation such that
there are no pairs of conjugate points.

In genus $g=1$, this problem is solved by Vélu's formula \cite{Velu}. Let $E_1$ be an elliptic curve defined over $K$ in Weierstrass form and let $G$ be a finite subgroup of $E_1$
of cardinality a prime number $\ell$. Then the elliptic curve $E_2=E_1/G$ is isogenous to $E_1$. A point $P$ of $E_1$ is sent by the isogeny $f:E_1\to E_2$ to the point with coordinates
\[x(f(P))=x(P)+\sum_{Q\in G\bb \{0\}}x(P+Q)-x(Q),\quad y(f(P))=y(P)+\sum_{Q\in G\bb \{0\}}y(P+Q)-y(Q).\]
Then, using the addition formula, it is possible to obtain the equation of $E_2$ in the Weierstrass form and a rational fraction $F$
such that the isogeny is $f:(x,y)\in E_1\mapsto (F(x),yF'(x))\in E_2$. See also \cite[Section 4.1]{MR2398793}.

For $g\ge 2$, a first generalization has been done by Cosset, Lubicz and Robert in \cite{CossetRobert,LubRob,LubRob2}. The authors explain how to compute separable
isogenies between principally polarized abelian varieties $A$ and $A/\mathcal{V}$ of dimension $g$ with a complexity of $\tilde{O}(\ell^{\frac{rg}2})$ operations in $K$, where $r=2$ if the odd prime number
$\ell$, different from the characteristic of $K$, is a sum of two squares and $r=4$ otherwise. For the former, the complexity is quasi-optimal since
it is quasi-linear in $\ell^g$, the degree of the isogeny (the cardinality of the maximal isotropic subgroup $\mathcal{V}$ of $A[\ell]$).
Here, the abelian varieties are represented through their theta null points.
A Magma \cite{Magma} package, AVIsogenies \cite{Avi}, implementing the ideas of these papers is available but the implementation concerns only the dimension $2$ case, that is,
generically, Jacobians of genus $2$ curves.

Note that for $g=2$ and $\ell=2$, the isogenies between Jacobians of curves can be computed using the Richelot
construction (see \cite[Chapter 9]{Prolegomena}). 
An algebraic-geometric approach for $g=2$ has been introduced by Dolgachev and Lehavi in \cite{DolgLeh} where the authors give
an effective algorithm for $\ell=3$ only. This approach is simplified and made more explicit in \cite{SmithDolgLeh}
resulting in an efficient algorithm for $\ell=3$.
For $g=3$ and $\ell=2$, there exists an algorithm \cite{SmithDLP} computing some of the possible $(2,2,2)$-isogenies from the Jacobian
of a hyperelliptic curve of genus $3$ over a finite field of characteristic $>3$, using Recillas' trigonal construction \cite{Recillas}
 and another algorithm \cite{LehRitz} computing the isogenous curve.

Another generalization, which is the starting point of the present paper, has been introduced in \cite{CouvEz}. In this paper, Ezome and Couveignes first explain how to define and compute functions 
$\eta$ and $\eta_f$ ($f$ is the isogeny of degree $\ell^g$)
from $J_{\mathcal{C}}$ to $\bar{K}$
for any genus $g\ge 2$, for any field $K$ of
characteristic $p$ not equal to $\ell$ or $2$, and for $\ell$ an odd prime.
Note that the $\eta_f$ function is a function $J_{\mathcal{C}}\to\bar{K}$ invariant by $\mathcal{V}$.
Then they focus on the genus $2$ case over finite fields.
The computation of the equation of an isogenous curve $\mathcal{D}$ is done in two steps.
First we define a map $\phi$ from $J_{\mathcal{C}}$ to the Kummer surface of $J_{\mathcal{D}}=J_{\mathcal{C}}/\mathcal{V}$ viewed in $\mathbb{P}^3$
using $\eta_f$ functions and compute the image in $\mathbb{P}^3$ of the $2$-torsion points of $J_{\mathcal{C}}$ using the geometry of Kummer surfaces.
Then according to this geometry, the intersection of a trope (a particular hyperplane) with the Kummer surface is a conic
containing the image of exactly $6$ $2$-torsion points corresponding to the $6$ Weierstrass points of $\mathcal{D}$.
Finally they explain how to describe the isogeny through rational fractions of degrees in $O(\ell)$
allowing one to compute the image of a point of $J_{\mathcal{C}}$ in $J_{\mathcal{D}}$ by the isogeny.
The resulting algorithm is quasi-linear in the degree $\ell^2$ of the isogeny (independently on the writing of $\ell$ as a sum of squares).

In this paper, we recall in Section \ref{sec:deffEta} the definition of the $\eta$ and $\eta_f$ functions and the algorithms
of \cite{CouvEz} to evaluate them.
These algorithms are quasi-linear in the degree of the isogeny and we do not present a theoretical improvement of them.
Thus these functions are seen as building blocks and we try to reduce as much as possible the number of times we evaluate them to reduce the
practical time of computation.

In Section \ref{sec:kumm}, we describe the particular geometry of Kummer varieties, which admit a $(m,n)$-configuration when
seen in $\mathbb{P}^{2^g-1}$, that is a set of $m$ hyperplanes (the tropes) and $m$ points (the image in $\mathbb{P}^{2^g-1}$ of the $2$-torsion points) such that each hyperplane contains $n$ of the $m$ points
and each of the $m$ points is contained in exactly $n$ hyperplanes.

Then in Section \ref{sec:hyp}, in the genus $2$ case, we use the $(16,6)$-configuration to compute the image of the $2$-torsion
points in $\mathbb{P}^3$, and explain how to deduce  the equation of $\mathcal{D}$ from it
and how to optimize this computation. We prove that the knowledge of the equation of the quartic describing the Kummer surface 
(which is used in \cite{CouvEz})
 is not necessary and use only $11$ evaluations of $\eta_f$ functions. This is a practical improvement compared to \cite{CouvEz} as the
 computation of the equation of the quartic requires in general around $140$ evaluations of $\eta_f$ functions.
We then turn to genus $3$. Here the curves are either hyperelliptic or
non-hyperelliptic (in which case they can be viewed as plane quartics).  
These two cases must be treated differently. We describe how the genus $2$ method can be naturally extended to the case where  $\mathcal{D}$
is hyperelliptic, using the $(64,29)$-configuration of the Kummer threefold. We focus on genus $3$ but it is clear that similar results exist
for $g>3$.

In Section \ref{sec:eqiso}, we recall the definition of the rational fractions needed to describe the isogeny and how to
compute them following \cite{CouvEz}, except for one step which is not practical.
Indeed, this step requires the computation of many algebraic relations between the $9$ functions forming a basis of
$H^0(J_{\mathcal{C}}/\mathcal{V},\mathcal{O}_{J_{\mathcal{C}}/\mathcal{V}}(3\mathcal{Y}))$, where $\mathcal{Y}$ is an effective divisor on $J_{\mathcal{C}}/\mathcal{V}$ associated to a  principal polarization of $J_\mathcal{C}/\mathcal{V}$.
We give another solution based on a good model of the Kummer surface allowing one to compute the
pseudo-addition law and to lift a point of the Kummer to the Jacobian.
We extend all these results in the genus $3$ case, where this pseudo-addition law has been recently described in \cite{Stoll}.

In Section \ref{sec:compatibility}, we construct algebraic theta functions as functions satisfying the same algebraic relations
between the analytic theta functions. We use these algebraic theta functions to compute the equation of $\mathcal{D}$ in genus $2$ through
the description of its Rosenhain form by theta constants.
Then we focus on the generic genus $3$ case where $\mathcal{D}$ is non-hyperelliptic. 
We use theta based formulas coming from the theory
of the reconstruction of a plane quartic from its bitangents.

Finally, Section \ref{sec:impl} is about our implementation. 

\section{Evaluation of the $\eta$ and $\eta_f$ functions}\label{sec:deffEta}
In this section, we recall the definition of the $\eta$ and $\eta_f$ functions of \cite{CouvEz} (note that we replace the name $\eta_X$ by $\eta_f$: $X$
denotes a divisor while $f$ denotes the isogeny, see Subsection \ref{subsec:etaf}).
We use the same notation as in that paper
and refer to it for more details.
\subsection{Definitions}\label{subsec:def}
This is \cite[Section 2.1]{CouvEz}.  
Let $\mathcal{C}$ be a projective, smooth, absolutely integral curve of genus $g\ge 2$ over a field $K$. 
We denote by $\Pic(\mathcal{C})$ its Picard group, $\Pic^d(\mathcal{C})$ the component of the Picard group
of linear equivalence classes of divisors (formal sums of points of $\mathcal{C}$) of degree $d$  and $J_{\mathcal{C}}:=\Pic^0(\mathcal{C})$ the Jacobian variety of $\mathcal{C}$.
If $D$ is a divisor on $\mathcal{C}$, then we denote by $\iota(D)$ its linear equivalence class. 

Let $W\subset \Pic^{g-1}(\mathcal{C})$ be the algebraic set representing classes of effective divisors of degree $g-1$.
The \emph{theta characteristics} are the $K$-rational points $\theta$ in $\Pic^{g-1}(\mathcal{C})$ such that $2\theta=\omega$, where $\omega$ denotes
the canonical divisor class. The difference of any pair of theta characteristics is a 2-torsion point in $J_{\mathcal{C}}$.
The translate $W_{-\theta}$ of $W$ by $-\theta$ is a divisor on $J_{\mathcal{C}}$. 
If $D$ is any effective divisor of $\mathcal{C}$ of degree $g-1$, then, by the Riemann--Roch theorem on effective divisors, $\ell(D)=\ell(\Omega-D)\ge 1$,
with $\Omega$ a divisor in the linear class of $\omega$
(see \cite[Chapter 2.3, Pages 244--245]{GH78}).
This implies that $[-1]^*W=W_{-\omega}$ and we deduce from this that \begin{equation}[-1]^*W_{-\theta}=W_{-\theta}.\end{equation}
Thus, $W_{-\theta}$  is said to be a \emph{symmetric} divisor on $J_{\mathcal{C}}$.

Consider now any $K$-point $O$ on $\mathcal{C}$, whose linear equivalence class is $o=\iota(O)$ in $\Pic^1(\mathcal{C})$.
The translate $W_{-(g-1)o}$ of $W$  by $-(g-1)o$ is a divisor on $J_{\mathcal{C}}$
but not necessarily a symmetric one. Taking $\vartheta=\theta-(g-1)o\in J_{\mathcal{C}}(K)$ we can construct a symmetric divisor
\begin{equation}\label{eq:varthetasym}[-1]^*W_{-(g-1)o-\vartheta}=W_{-(g-1)o-\vartheta}.\end{equation}
Let $I$ be a positive integer, $e_1,\ldots,e_I\in \Z$ and $u_1,\ldots,u_I\in J_{\mathcal{C}}(\bar{K})$. The formal sum $\frak{u}=\sum_{1\le i\le I}e_i[u_i]$ is a zero-cycle on $J_{\mathcal{C},\bar{K}}$.
Define the \emph{sum} and \emph{degree}  functions of a zero-cycle by
\begin{equation}s(\frak{u})=\sum_{1\le i\le I}e_iu_i\in J_{\mathcal{C}}(\bar{K})\qquad\textrm{and}\qquad \deg(\frak{u})=\sum_{1\le i \le I}e_i\in\Z.\end{equation}
Let $D$ be a divisor on $J_{\mathcal{C},\bar{K}}$.
The divisor $\sum_{1\le i \le I}e_iD_{u_i}-D_{s(\frak{u})}-(\deg(\frak{u})-1)D$ is principal
(\cite[Chapter III, Section 3, Corollary 1]{LangAV})
so it defines a function up to a multiplicative constant.
To fix this constant, we choose a point $y\in J_{\mathcal{C}}(\bar{K})$ and consider the function whose evaluation at $y$ is $1$. This implies that we want $y$  not to be in the support of this divisor. This unique function is denoted by $\eta_D[\frak{u},y]$. To summarize, $\eta_D[\frak{u},y]$ satisfies
\begin{equation}\label{DefEta}(\eta_D[\frak{u},y])= \sum_{1\le i \le I}e_iD_{u_i}-D_{s(\frak{u})}-(\deg(\frak{u})-1)D\qquad \textrm{and}\qquad \eta_D[\frak{u},y](y)=1.\end{equation}
We will sometimes denote by $\eta_D[\frak{u}]$ the function defined up to
a multiplicative constant.
Moreover, we have the following additive property
\begin{equation}\label{EqAdd} \eta_D[\frak{u}+\frak{v},y]=\eta_D[\frak{u},y]\cdot\eta_D[\frak{v},y]\cdot\eta_D[[s(\frak{u})]+[s(\frak{v})],y],\end{equation}  
which can be proved by comparing divisors.

We are mainly interested in the cases where $D=W_{-(g-1)o}$ or $D=W_{-(g-1)o-\vartheta}=W_{-\theta}$. Note that we have
\begin{equation}\label{eq:evaltheta} \eta_{W_{-\theta}}[\frak{u},y](x)=\eta_{W_{-(g-1)o}}[\frak{u},y+\vartheta](x+\vartheta)\end{equation}
so that we will focus on the first divisor; and to simplify the notations, we write $\eta[\frak{u},y]$ instead of $\eta_{W_{-(g-1)o}}[\frak{u},y]$.

\subsection{Evaluation of $\eta[\frak{u},y]$}\label{subsec:evaleta}
Fix $\frak{u}=\sum_{1\le i\le I} e_i[u_i]$ a zero-cycle on $J_{\mathcal{C}}$ with $u_i\in J_{\mathcal{C}}(K)$ for $1\le i\le I$, $y\in J_{\mathcal{C}}(K)$ not in the support of $\eta[\frak{u}]$ and  $x\in J_{\mathcal{C}}(K)$ not in the support
of $\eta[\frak{u},y]$.
Assume that $x=\iota(D_x-gO)$ and $y=\iota(D_y-gO)$ where $D_x$ and $D_y$ are effective divisors of the curve $\mathcal{C}$ of degree $g$ not having $O$ in their support (this is the generic case).
Write $D_x=X_1+\ldots+X_g$ and $D_y=Y_1+\ldots+Y_g$. This writing is unique (see \cite[Section 2.6]{Claus}).
Make also the assumption that $\deg(\frak{u})=0\in \Z$ and $s(\frak{u})=0\in J_{\mathcal{C}}(K)$. This is not a restriction because if $\frak{u}$ does not satisfy these properties, then
the zero-cycle $\frak{u'}=\frak{u}-[s(\frak{u})]-(\deg(\frak{u})-1)[0]$ does and the functions $\eta[\frak{u}]$ and $\eta[\frak{u'}]$ have the same divisor.

The computation of $\eta[\frak{u},y](x)$ goes as follows. See \cite[Section 2]{CouvEz} for the details.

\begin{enumerate}
\item For every $1\le i \le I$, find an effective divisor $D^{(i)}$ of degree $2g-1$ such that $D^{(i)}$ meets neither $D_x$ nor $D_y$ and $\iota(D^{(i)})-\omega-o$
  is the class $u_i$. Taking $U_i-gO$ in the class of $u_i$, where $U_i$ is effective of degree $g$, and taking a canonical divisor $\Omega$ on $\mathcal{C}$, the divisor $D^{(i)}$ can be found looking  
  at the Riemann--Roch space $\mathcal{L}(U_i-(g-1)O+\Omega)$.
  The condition on the degree of $D^{(i)}$ and the Riemann--Roch theorem say that $\ell(D^{(i)})=g$. 
\item Find a non-zero function $h$ in $K(\mathcal{C})$  with divisor $\sum_{1\le i \le I}e_i D^{(i)}$. This function exists thanks to the conditions on the zero-cycle. 
\item For every $1\le i\le I$,  compute a basis $f^{(i)}=(f_k^{(i)})_{1\le k\le g}$ of $\mathcal{L}(D^{(i)})$. This step and the previous one are an effective Riemann--Roch theorem. 
\item For every $1\le i\le I$, compute $\delta_x^{(i)}:=\det(f_k^{(i)}(X_j))_{1\le k,j\le g}$ and $\delta_y^{(i)}:=\det(f_k^{(i)}(Y_j))_{1\le k,j\le g}$.
\item Compute $\alpha[h](x):=\prod_{i=1}^gh(X_i)$ and $\alpha[h](y):=\prod_{i=1}^gh(Y_i)$.
\item Return $\frac{\alpha[h](x)}{\alpha[h](y)}\cdot\prod_{1\le i\le I}(\delta_x^{(i)}/\delta_y^{(i)})^{e_i}$ (which is equal to $\eta[\frak{u},y](x)$).
\end{enumerate}

In the case that $D_x$ (or $D_y$) is not simple, then the $\delta_x^{(i)}$ are zero and the product $\prod_{1\le i\le I}(\delta_x^{(i)})^{e_i}$ is not defined (some $e_i$ are negative)
while $\eta[\frak{u},y](x)$ is. This last value can be obtained considering the field $L=K((t))$ for a formal parameter $t$. Indeed, assume for example $D_x=nX_1+X_{n+1}+\ldots+X_g$ and $X_i\ne X_j$ if $i\ne j$. Fix a local parameter $z\in K(\mathcal{C})$ at $X_1$ and $n$ distinct scalars $(a_j)_{1\le j\le n}$ in $K$ (if $\# K$ is too small, then consider a small degree extension of it).
Denote by $X_1(t),X_2(t),\ldots,X_n(t)$
the points in $\mathcal{C}(L)$ associated to the values $a_1t,\dots,a_nt$ of the local parameter $z$.
Do the computations of the algorithm with $D_x(t)=X_1(t)+\ldots+X_n(t)+X_{n+1}+\ldots+X_g$ and set $t=0$ in the result. 
According to \cite[Section 2.6]{CouvEz}, the necessary $t$-adic accuracy  to obtain the good result is $g(g-1)/2$.

Denote by $\frak{O}$ a positive absolute constant. Any statement containing this symbol is true if this symbol is replaced by
a big enough real number. Similarly, denote by $\frak{e}(z)$ a real function in the real parameter $z$ belonging to the class $o(1)$.
\begin{theo} There exists a deterministic algorithm that takes as input
  \begin{itemize}
  \item a finite field $K$ with cardinality $q$;
  \item a curve $\mathcal{C}$ of genus $g\ge 2$ over $K$;
  \item a collection of $K$-points $(u_i)_{1\le i \le I}$ in the Jacobian $J_{\mathcal{C}}$ of $\mathcal{C}$;
  \item a zero-cycle $\frak{u}=\sum_{1\le i\le I}e_i[u_i]$ on $J_{\mathcal{C}}$ such that $\deg(\frak{u})=0$ and $s(\frak{u})=0$;
  \item a point $O$ in $\mathcal{C}(K)$;
  \item and two points $x,y\in J_{\mathcal{C}}(K)$ not in $\bigcup_{1\le i\le I} W_{-(g-1)o+u_i}$.
  \end{itemize}
  Let $|e|:=\sum_{1\le i\le I}|e_i|$.
The algorithm computes $\eta[\frak{u},y](x)$ in time $(g\cdot |e|)^\frak{O}\cdot(\log{q})^{1+\frak{e}(q)}$.
Using fast exponentiation and equation (\ref{EqAdd}), the complexity is $g^{\frak{O}}\cdot I\cdot (\log{|e|})\cdot(\log{q})^{1+\frak{e}(q)}$
but there exists a subset $\emph{FAIL}(K,\mathcal{C},\frak{u},O)$ of $J_{\mathcal{C}}(K)$ with density $\le g^{\frak{O}g}\cdot I\cdot\log(|e|)/q$ such that the algorithm
succeeds whenever neither  $x$ nor $y$ belongs to this subset.
\end{theo}


\paragraph{Fast multiple evaluation.}
For our applications, we need to evaluate $\eta[\frak{u},y]$ at many random points $x$ of the Jacobian to do linear algebra. So we could ask if there is some redundant computation. Obviously the values $\delta_x^{(i)}$ and $\alpha[h][x]$ are not the same for different points $x$, but what about $\delta_y^{(i)}$ and $\alpha[h][y]$ ?
The divisors $D^{(i)}$ at Step $1$ depend on $x$ so that $h$, the bases $f^{(i)}$ and thus $\delta_y^{(i)}$ and $\alpha[h][y]$ depend also on $x$. So there is no
redundant computation.
But if we do not consider the dependency of the $D^{(i)}$ on $x$ and take $D^{(i)}=U_i-(g-1)O+\Omega$ (for example)
aand we assume that the condition on $y$ does not interfere, then
$h$, the basis $f^{(i)}$, the values $\delta_y^{(i)}$ and $\alpha[h](y)$ can be computed once and for all in a precomputation step.
It remains to do Steps $4$, $5$ and $6$ for each $x$ and the points where the computation does not work are simply discarded.
The main advantage of this method is that the effective Riemann--Roch algorithms are only executed one time.
The random $x$ can be obtained by taking $g$ random points of the curve $\mathcal{C}$ so that we directly have the points $X_i$.

\subsection{Evaluation of $\eta_f[\frak{u},y]$} \label{subsec:etaf}

In the preceding subsection, we have defined functions on $J_{\mathcal{C}}$. Let $\mathcal{V}\subset J_{\mathcal{C}}[\ell]$ be a maximal isotropic subgroup for the $\ell$-Weil pairing.
We now introduce functions on $J_{\mathcal{C}}$ invariants by $\mathcal{V}$. This is \cite[Sections 4 and 5]{CouvEz}. 

Let $f:J_{\mathcal{C}}\to J_{\mathcal{C}}/\mathcal{V}$ be an $(\ell,\ldots,\ell)$-isogeny.
Let $\mathcal{L}=\mathcal{O}_{J_{\mathcal{C}}}(\ell W_{-\theta})$. According to \cite[Section 5]{CouvEz}, there exists a symmetric
principal polarization $\mathcal{M}$ on $J_{\mathcal{C}}/\mathcal{V}$ which satisfies $\mathcal{L}=f^*\mathcal{M}$. 
As $h^0(\mathcal{M})=1$, there exists an effective divisor $\mathcal{Y}$ on $J_{\mathcal{C}}/\mathcal{V}$ associated to $\mathcal{M}$.
The divisor $\mathcal{X}=f^*\mathcal{Y}$
is effective, linearly equivalent to $\ell W_{-\theta}$ and invariant under $\mathcal{V}$ (acting by translation).

We are interested in the function $\eta_{\mathcal{X}}[\frak{u},y]$ (see Equation (\ref{DefEta})) for some zero-cycle $\frak{u}=\sum_{1\le i\le I} e_i[u_i]$ in $J_{\mathcal{C}}$ and $y\in J_{\mathcal{C}}(K)$
with the usual restrictions.
Take $v_i=f(u_i)$ and let $\frak{v}=\sum_{1\le i \le I}e_i[v_i]$ be a zero-cycle on $J_{\mathcal{C}}/\mathcal{V}$. Consider the function
$\eta_{\mathcal{Y}}[\frak{v},f(y)]$ on $J_{\mathcal{C}}/\mathcal{V}$ having $\sum_{1\le i\le I}e_i\mathcal{Y}_{v_i}-\mathcal{Y}_{s(\frak{v})}-(\deg(\frak{v})-1)\mathcal{Y}$ as divisor
and taking value $1$ at $f(y)$.
Then $\eta_{\mathcal{Y}}[\frak{v},f(y)]\circ f$ is equal to $\eta_{\mathcal{X}}[\frak{u},y]$. So $\eta_{\mathcal{X}}[\frak{u},y]$ is invariant by $\mathcal{V}$ and can
be identified with $\eta_{\mathcal{Y}}[\frak{v},f(y)]$. This allows us to work on $J_{\mathcal{C}}/\mathcal{V}$
while staying in $J_{\mathcal{C}}$.
A point $z$ in $J_{\mathcal{C}}/\mathcal{V}$ is seen as a point $x$ in $J_{\mathcal{C}}$ such that $f(x)=z$.
To insist on the fact that $\eta_{\mathcal{X}}[\frak{u},y]$ is related to the isogeny, we name this function $\eta_f[\frak{u},y]$.

We want now to evaluate the function $\eta_f[\frak{u},y]$ at $x$. The trick consists in constructing a function $\Phi_{\mathcal{V}}$ having $\mathcal{X}-\ell W_{-\theta}$ as
divisor. Indeed, assuming that $s(\frak{u})=0$ and $\deg(\frak{u})=0$, the divisor of $\eta_f[\frak{u},y]$ is $\sum_{i=1}^Ie_i\mathcal{X}_{u_i}$ while the divisor
of $\prod_{1\le i \le I}\Phi_{\mathcal{V}}(x-u_i)^{e_i}$ is $\sum_{i=1}^Ie_i(\mathcal{X}_{u_i}-\ell W_{-\theta+u_i})$. To compensate, consider the function
$(\eta[\frak{u}](x+\vartheta))^\ell$ which has divisor $\ell\sum_{i=1}^Ie_iW_{-(g-1)o-\vartheta+u_i}=\ell\sum_{i=1}^Ie_iW_{-\theta+u_i}$ because
$\vartheta=\theta-(g-1)o$. Thus, looking at the evaluation at $y$, we deduce 
\begin{equation}\label{eq:defetax}
  \eta_f[\frak{u},y](x)=(\eta[\frak{u},y+\vartheta](x+\vartheta))^\ell\cdot\prod_{1\le i\le I}(\Phi_{\mathcal{V}}(x-u_i))^{e_i}\cdot \prod_{1\le i\le I}(\Phi_{\mathcal{V}}(y-u_i))^{-e_i}.
\end{equation}

The construction and computation of $\Phi_{\mathcal{V}}$ are as follows. For any $w\in\mathcal{V}$, define $w':=\frac{\ell+1}2\cdot w$ ($\ell$ must be odd).
Fix $\phi_u,\phi_y\in J_{\mathcal{C}}(K)$ and consider the functions 
\[\theta_w(x)=\eta[\ell[w']-\ell[0],w'-x+\vartheta](x-w'+\vartheta),\]
\[\tau[\phi_u,\phi_y](x)=\eta[[(\ell-1)\phi_u]+(\ell-1)[-\phi_u]-\ell[0],\phi_y+\vartheta](x+\vartheta),\]
and \[a_w(x)=\theta_w(x)\cdot \tau[\phi_u,\phi_y](x-w).\]
Then we can define $\Phi_{\mathcal{V}}$ as  \[\Phi_{\mathcal{V}}(x)=\sum_{w\in\mathcal{V}}a_w(x).\]
This is also equal to $\sum_i\tr_{L_i/K}(a_{w_i}(x))$ if the subgroup $\mathcal{V}$ is given by a collection of fields
extensions $(L_i/K)$ and points $w_i\in\mathcal{V}(L_i)$ such that $\mathcal{V}$
is the disjoint union of 
the sets containing $w_i$ and all its conjugates.

As $\#\mathcal{V}=\ell^g$, the number of calls to the $\eta$ function to compute $\eta_f$ is bounded by
$1+4\cdot I\cdot \ell^g$.

\begin{theo}\label{theo:evaletaf} There exists a deterministic algorithm that takes as input
  \begin{itemize}
  \item  a finite field $K$ with characteristic $p$ and cardinality $q$;
  \item a curve $\mathcal{C}$ of genus $g\ge 2$ over $K$;
  \item a zero-cycle $\frak{u}=\sum_{1\le i\le I}e_i[u_i]$ in the Jacobian $J_{\mathcal{C}}$ of $\mathcal{C}$ such that
    $u_i\in J_{\mathcal{C}}(K)$ for every $1\le i\le I$, $\deg(\frak{u})=0$ and $s(\frak{u})$;
  \item a theta characteristic $\theta$ defined over $K$;
  \item an odd prime number $\ell\ne p$;
  \item a maximal isotropic $K$-subgroup scheme $\mathcal{V}\subset J_{\mathcal{C}}[\ell]$;
  \item two classes $x$ and $y$ in $J_{\mathcal{C}}(K)$ such that $y\not\in (\bigcup_iW_{-\theta+u_i})\cup(\bigcup_i\mathcal{X}_{u_i})$.   
  \end{itemize}
  The algorithm returns \emph{FAIL} or $\eta_f[\frak{u},y](x)$ in time $I\cdot (\log{|e|})\cdot g^{\frak{O}}\cdot (\log{q})^{1+\frak{e}(q)}\cdot \ell^{g(1+\frak{e}(\ell^g))}$,
  where $|e|=\sum_{1\le i\le I}|e_i|$. For given $K$, $\mathcal{C}$, $\frak{u}$, $\theta$, $\mathcal{V}$, there exists a subset
  $\emph{FAIL}(K,\mathcal{C},\frak{u},\theta,\mathcal{V}$) of $J_{\mathcal{C}}(K)$ with density
  $\le I\cdot (\log{|e|})\cdot g^{\frak{O}g}\cdot \ell^{g^2}\cdot (\log{\ell})/q$ and such that the algorithm succeeds whenever none of $x$ and $y$ belongs
  to this subset.
\end{theo}

\paragraph{Fast multiple evaluation.}

We need to evaluate $\eta_f[\frak{u},y]$ at many random points $x$ and we ask if there is some redundant computation.
We use the same idea as for the $\eta[\frak{u},y]$ functions in the previous subsection to minimize the number of times
we use effective Riemann--Roch algorithms.
Thus
\begin{itemize}
\item The product $\prod_{1\le i \le I}(\Phi_{\mathcal{V}}(y-u_i))^{-e_i}$ does not depend on $x$ anymore;
\item We take the same $\phi_u$ and $\phi_y$ in $J_{\mathcal{C}}$ for all $x$;
\item Then we can do the fast multiple evaluation of the $\eta$ functions $\tau[\phi_u,\phi_y]$, $\eta[\frak{u},y+\vartheta]$,
  and $\theta_w$ for all $w\in\mathcal{V}$.
\end{itemize}

\section{Geometry of the Kummer variety}\label{sec:kumm}

We assume $char(K)\ne 2$  
and $K$ algebraically closed.

Let $a\in J_{\mathcal{C}}[2]$ and $y\in J_{\mathcal{C}}$. The function $\eta_{W_{-\theta}}[2[a]-2[0],y]$ 
whose divisor is $2(W_{-\theta+a}-W_{-\theta})$
is said to be a \emph{level $2$  function}. The level $2$  functions generate the space 
of functions  $H^0(J_{\mathcal{C}},\mathcal{O}_{J_{\mathcal{C}}}(2W_{-\theta}))$, which is of dimension $2^g$.
Let  $\eta_1,\ldots,\eta_{2^g}$ be a basis of this space consisting in level $2$ functions.
The map $\phi=(\eta_1:\ldots:\eta_{2^g}):J_{\mathcal{C}}\to\mathbb{P}^{2^g-1}$
factors through the projection $J_{\mathcal{C}}\to J_{\mathcal{C}}/\langle \pm 1\rangle$ and a
morphism $J_{\mathcal{C}}/\langle \pm 1\rangle\to\mathbb{P}^{2^g-1}$, which is a closed embedding (\cite[Proposition 2.3]{DolgLeh}).
The \emph{Kummer variety} of $J_{\mathcal{C}}$ is $J_{\mathcal{C}}/\langle \pm 1\rangle$. We identify it with its image in $\mathbb{P}^{2^g-1}$.
According to \cite[Proposition 3.1]{Muller}, the Kummer variety can always be described by an intersection of quartics.
We consider the following.



\begin{itemize}
\item If $g=2$,  $1$ quartic is enough;
\item If $g=3$ and the curve is hyperelliptic,  $1$ quadric and $34$ quartics are needed (\cite[Theorem 2.5]{Stoll} extending \cite[Theorem 3.3]{Muller})
  (in our case, we have always computed a lot of equations and reduced them in computing a Gröbner basis and this yielded $1$ quadric and $35$ quartics);
\item If $g=3$ and the curve is  non-hyperelliptic, it is possible, instead of quartics, to describe the Kummer variety by cubics equations (\cite[Theorem 7.5]{brambila}).
  We have always found $8$ cubics equations.
\end{itemize}  

These equations can be computed in evaluating $\eta_1$, $\ldots$, $\eta_{2^g}$
at random points and then doing linear algebra. In genus $2$, the basis is of cardinality $4$ and a quartic has at most $35$ coefficients, so that
the number of evaluations is at least $35\times 4$. The number of coefficients is $330$ in genus $3$ for quartics and $36$ for quadrics 
implying at least $330\times 8$ evaluations because the basis has $8$ elements,
but in the non-hyperelliptic case, as the Kummer variety can be described by cubics, $120\times 8$ evaluations are needed.
(Recall that the number of monomials of degree $d$ with $v$ variables is ${v+d-1}\choose{d}$).

\begin{remm}Having these equations help the computation of the isogeny but they are not necessary. 
Computing them does not impact the complexity of the algorithms but have a huge impact on the practical computations (see next subsection).
\end{remm}

\begin{remm} In this section we are describing  the Kummer variety of a curve $\mathcal{C}$.
  Then to compute the isogeny we choose a basis of $\eta_f$ functions so that we end up with the Kummer variety of the isogenous curve~$\mathcal{D}$.
\end{remm}

Using linear algebra, it is possible to write  $\eta_{W_{-\theta}}[2[a]-2[0],y]$ as a linear combination of the basis: we have
$\eta_{W_{-\theta}}[2[a]-2[0],y]=\sum_{i=1}^{2^g}c_i\eta_i$ with $c_i\in K$.
Call $Z_1,\ldots,Z_{2^g}$ the projective coordinates associated to the basis. 
This gives an equation $Z_a=\sum_{i=1}^{2^g}c_iZ_{i}$ and the equation $Z_a=0$ is the image of $W_{-\theta+a}$ in the Kummer variety seen in $\mathbb{P}^{2^g-1}$.
\begin{deff}
  The hyperplanes $Z_a=0$ for $a\in J_{\mathcal{C}}[2]$ are called \emph{singular planes} or \emph{tropes}.
The image of the $2$-torsion points in $\mathbb{P}^{2^g-1}$ are called \emph{singular points} or \emph{nodes}.
  \end{deff}

The set of tropes with the set of nodes form a configuration.

\begin{deff}
  A $(m,n)$-configuration in $\mathbb{P}^N$ is the data of $m$ hyperplanes and $m$ points such that each hyperplane contains $n$ points and each point is contained in $n$ hyperplanes.
\end{deff}


The configuration can be described through a symplectic basis of the $2$-torsion: let $e_1$, \ldots, $e_g$, $f_1$, \ldots, $f_g$ be such a basis. We represent an  element
$a=\epsilon_1e_1+\ldots+\epsilon_ge_g+\rho_1f_1+\ldots+\rho_g f_g$ by the matrix
$\begin{ppsmallmatrix}\epsilon_1&\ldots&\epsilon_g\\\rho_1&\ldots&\rho_g\end{ppsmallmatrix}$ and we define its \emph{characteristic} as $\sum_{i=1}^g\epsilon_i\rho_i$.

  Kummer surfaces have been thoroughly studied (see \cite[Section 10.2]{Birk} for example when $K=\mathbb{C}$) and their
 $(16,6)$-configuration is a corollary of the following proposition.
  \begin{prop}\label{prop:confG2} 
    Let $\mathcal{C}$ be a genus~$2$ curve.
    There is a $2$-torsion point $a_0$ such that for any $a'=\begin{ppsmallmatrix} \epsilon_{11}&\epsilon_{12}\\ \rho_{11}&\rho_{12} \end{ppsmallmatrix}$  and
  $a''=\begin{ppsmallmatrix} \epsilon_{21}&\epsilon_{22}\\ \rho_{21}&\rho_{22} \end{ppsmallmatrix}$ in $J_{\mathcal{C}}[2]$, the following conditions are equivalent
\begin{itemize} 
\item the image of $a'$ in $\mathbb{P}^3$ 
  is contained in the trope $Z_{a_0+a''}$; 
\item either $((\epsilon_{11},\rho_{11})=(\epsilon_{21},\rho_{21})$ and $(\epsilon_{12},\rho_{12})\ne(\epsilon_{22},\rho_{22}))$ or $((\epsilon_{11},\rho_{11})\ne (\epsilon_{21},\rho_{21})$ and $(\epsilon_{12},\rho_{12})=(\epsilon_{22},\rho_{22}))$.
  \end{itemize}
\end{prop}  
  \begin{proof}    
    First note that it is easy to prove that there is a $(16,6)$-configuration. A genus $2$ curve has $6$
  Weierstrass points $r_1$, \ldots, $r_6$ from which we deduce the $16$ $2$-torsion points. For $i\in\{1,\ldots,6\}$ and $a\in J_{\mathcal{C}}[2]$, the 
  $r_i-\theta+a$ are the only $2$-torsion points in $W_{-\theta+a}$, where we identify the points in $\mathcal{C}$ with the points in $\Pic^1$.
  Moreover, the $2$-torsion point $r_j-\theta+a$ is in $W_{-\theta+a+(r_i-\theta)+(r_j-\theta)}$ for $i\in\{1,\ldots,6\}$.
  
  This proposition is exactly \cite[Proposition 10.2.5]{Birk} where $K=\mathbb{C}$. We generalize the proof of this reference, with the tools of \cite[Section 2]{MumEqAVI}
  (in particular, the properties of $e_*^{\mathcal{L}}$ page $304$, Proposition~$2$ and Corollary~$1$).

  Let $\mathcal{L}$ be the symmetric line bundle defining the principal polarization and $e_1$, $e_2$, $f_1$, $f_2$ the symplectic basis of the $2$-torsion (for the Weil pairing).
  Let $V_1=\langle e_1,e_2 \rangle$, $V_2=\langle f_1,f_2 \rangle$.
  We look for a symmetric line bundle $\mathcal{L}_0$ which is a translate of $\mathcal{L}$ such that $e_*^{\mathcal{L}_0}(x)=1$ for $x\in\{e_1,e_2,f_1,f_2\}$.
  This line bundle exists because:
  \begin{itemize}
  \item There are $2^4=16$ values possible for $(e_*^{\mathcal{L}_0}(x))_{x\in\{e_1,e_2,f_1,f_2\}}$ and there also are $16$ $2$-torsions points.
  \item Each $e_*^{t_a^*\mathcal{L}}$ is different thanks to the non-degeneracy of the commutator pairing $e^{\mathcal{L}^2}$ (\cite[Theorem 1]{MumEqAVI}). See also Section \ref{subsec:alg} for the definition of this pairing.
  \item We have  $\mathcal{L}_0=\mathcal{L}\otimes\mathcal{M}$, where $\mathcal{M}=t_a^*\mathcal{L}\otimes\mathcal{L}^{-1}$ for some $a\in J_{\mathcal{C}}[2]$, from which we deduce
  that $e^{\mathcal{L}^2}=e^{\mathcal{L}_0^2}$ because $M\in \mathrm{Pic}^0(J_{\mathcal{C}})$. This pairing is the Weil pairing associated to the principal polarization.
  \end{itemize}  
  Then for $x=v_1+v_2\in J_{\mathcal{C}}[2]$, $v_i\in V_i$, we have that $e_*^{\mathcal{L}_0}(x)=e_*^{\mathcal{L}_0}(v_1+v_2)=e^{\mathcal{L}_0^2}(v_1,v_2)$.
  Thus $e_*^{\mathcal{L}_0}$ satisfies \cite[Equation (3) page 47]{Birk} on $2$-torsion points.
  We also have that $e_*^{\mathcal{L}}(x)=(-1)^{m(x)-m(0)}$, which is similar to \cite[Proposition 4.7.2]{Birk}.
  So we can do the same proof of \cite[Proposition 10.2.5]{Birk} with $e_*^{\mathcal{L}_0}$ instead of $\chi_0$.  
\end{proof}

  \begin{remm}
    In the analytic theta function theory (see Section \ref{sec:compatibility}), we have that the six theta constants having odd characteristic
    are equal to $0$. Here, note that the image of a $2$-torsion point $a'$ is in the trope $Z_{a_0+a''}$ for some $a''$ when the characteristic of
    $a'+a''+\begin{ppsmallmatrix}1&1\\1&1\end{ppsmallmatrix}$ is odd.
    The shift by $\begin{ppsmallmatrix}1&1\\1&1\end{ppsmallmatrix}$ comes from an arbitrary choice of
      a line bundle in the proof of this proposition  in \cite{Birk}. Another choice would give another description of the configuration. 
  \end{remm}

\begin{remm}
  Note that the trope $Z_a$ for $a=0$ contains the image of the six $2$-torsion points $a_i=r_i-\theta$.
  This trope correspond to the case $a''=a_0$.
  Knowing the matrices associated to these six points (for a fixed symplectic basis),
  we can compute  $a_0$ using the second point of the proposition.
  Moreover, we deduce that $a_0\not\in\{r_i-\theta\}_{i\in\{1,\ldots,6\}}$.
  See also Remark \ref{rem:calcula0} for the computation of $a_0$. 
\end{remm}

\begin{coro}\label{cor:inter2} Any two different tropes have exactly two nodes in common.\end{coro} 
\begin{proof} This is a quick consequence of the previous proposition. See \cite[Corollary 10.2.8]{Birk} (when $K=\mathbb{C}$ but the proof in general is the same).
\end{proof} 

In genus $3$, there is a $(64,28)$-configuration for hyperelliptic and non-hyperelliptic curves and for the former, the configuration can be extended to
a $(64,29)$-configuration.

\begin{prop}\label{prop:confG3} Let $\mathcal{C}$ be a genus $3$ curve. There exists a $2$-torsion point $a_0$ such that for all
  $a'=\begin{ppsmallmatrix} \epsilon_{11}&\epsilon_{12}&\epsilon_{13}\\ \rho_{11}&\rho_{12}&\rho_{13} \end{ppsmallmatrix}$  and
$a''=\begin{ppsmallmatrix} \epsilon_{21}&\epsilon_{22}&\epsilon_{23}\\ \rho_{21}&\rho_{22}&\rho_{23} \end{ppsmallmatrix}$ in $J_{\mathcal{C}}[2]$,
the image of the point $a'$ in $\mathbb{P}^7$ is contained in $Z_{a_0+a''}$ if and only if one of the following conditions is satisfied
\begin{itemize}
\item $a'=a''$;
\item $(\epsilon_{11},\rho_{11})=(\epsilon_{21},\rho_{21})$ and $(\epsilon_{12},\rho_{12})\ne (\epsilon_{22},\rho_{22})$ and $(\epsilon_{13},\rho_{13})\ne (\epsilon_{23},\rho_{23})$;
\item $(\epsilon_{11},\rho_{11})\ne (\epsilon_{21},\rho_{21})$ and $(\epsilon_{12},\rho_{12})= (\epsilon_{22},\rho_{22})$ and $(\epsilon_{13},\rho_{13})\ne (\epsilon_{23},\rho_{23})$;
\item $(\epsilon_{11},\rho_{11})\ne (\epsilon_{21},\rho_{21})$ and $(\epsilon_{12},\rho_{12})\ne (\epsilon_{22},\rho_{22})$ and $(\epsilon_{13},\rho_{13})= (\epsilon_{23},\rho_{23})$;
\item $a'=a''+a_0+2r-\theta$ and $\mathcal{C}$ is hyperelliptic, where $r$ is any Weierstrass point, seen in $\Pic^1(\mathcal{C})$.
\end{itemize}
\end{prop}
\begin{proof}
  We have obtained this result in adapting the proof of \cite[Proposition 10.2.5]{Birk}, using   $\begin{ppsmallmatrix}1&1&1\\1&1&1\end{ppsmallmatrix}$.
    In the hyperelliptic case, the divisor $W_{-\theta}$ contains $29$ points of $2$-torsion
    (coming from the combination of two Weierstrass points among the $8$,  giving us ${{8}\choose{2}}+1=29$)
    against $28$ in the non-hyperelliptic case (coming from the $28$ bitangents, see Section \ref{subsec:bitangent}).
    Among the $29$ points, exactly one is such that the multiplicity of $W_{-\theta}$ at this point is even: this is the point $2r$
    in $W$ and thus $2r-\theta$ is in $W_{-\theta}$ (recall that if $r_1$, $r_2$ are linear classes of divisors in $\Pic^1(\mathcal{C})$
    coming from Weierstrass points, then $2r_1\sim 2r_2$).   
\end{proof}

\begin{remm}   It is well-known that the $28$ analytic theta constants having odd characteristic are equal to $0$ and that a genus $3$ curve is hyperelliptic
  if and only if (exactly) one even theta constant is equal to $0$. 
\end{remm}

\begin{remm} As in the genus $2$ case, we can compute $a_0$ knowing the points
  in $W_{-\theta}$. Note that because of the first condition, we have that  $a_0$ is one of the $2$-torsion points in $W_{-\theta}$.
  \end{remm}

 \section{Computation of the equation of the curve in the hyperelliptic case} \label{sec:hyp}
If the genus of $\mathcal{C}$ is $2$, then the quotient $J_{\mathcal{C}}/\mathcal{V}$ is generically the Jacobian of a   
genus $2$  curve $\mathcal{D}$,
while if it is $3$, this is generically the Jacobian of a genus $3$ curve $\mathcal{D}$, which can be hyperelliptic or non-hyperelliptic (a plane quartic) and the latter
is the generic case.
The aim of this section is to compute a model of $\mathcal{D}$ when $\mathcal{D}$ is hyperelliptic of genus $2$ or $3$ in using the geometry of the Kummer
variety.

We give a general method to achieve this based on \cite{CouvEz}. We optimize it and show that in genus $2$ we can obtain $\mathcal{D}$ in $11$ evaluations
of $\eta_f$ functions. We give a solution in genus $3$ without optimizing the number of evaluations.

\paragraph{Notations.}
Let $\mathcal{C}$ be a hyperelliptic curve of genus $g\in\{2,3\}$ over the algebraic closure (to simplify the exposition) of a finite field $K$
of characteristic $\ne 2$.
We assume that the curve is given by
an \emph{imaginary model}  so that we have $\mathcal{C}:Y^2=h_{\mathcal{C}}(X)$ for $h_{\mathcal{C}}$  of degree $2g+1$ 
and having a unique point at infinity $O$.
Coming back to the notations of Section \ref{subsec:def},  the $K$-point we choose is $O$. Then $2O$ is a canonical divisor and we take
the theta characteristic $\theta=\iota(O)=o$; then $\vartheta=0\in J_{\mathcal{C}}$. We use the $\eta$ and $\eta_f$ functions defined by the divisor $W_{-o}$.
Let $r_1,\ldots,r_{2g+2}$ be the $2g+2$ Weierstrass points of the curve $\mathcal{C}$, where $r_{2g+2}$ is $O$.
By an abuse of notation, we also denote $r_i$ the class of $r_i$ in $\Pic^1(\mathcal{C})$. The $2$-torsion points in $J_{\mathcal{C}}$ for $g=2$ are 
$a_i:=r_i-o$ for $i\in\{1,\ldots, 2g+2\}$ and $a_{ij}:=r_i+r_j-2o$ for $1\le i,j\le 2g+1$ ($i\ne j$). Moreover if $g=3$, we add $a_{ijk}:=r_i+r_j+r_k-3o$ for
$1\le i,j,k \le 2g+1$ (and $i,j,k$ distincts).
Fix $y\in J_{\mathcal{C}}$.

\subsection{Computing the equation of the isogenous curve in genus 2}\label{subsec:isoG2}

\paragraph{The (16,6)-configuration.}
Consider the level $2$ functions $\eta_f[2[a]-2[0],y]$ for all $a\in J_{\mathcal{C}}[2]$.
Looking at divisors, we have for a $2$-torsion point $a\ne a_6$ that $\eta_f[2[a]-2[0],y](x)=0$ for the values $x\in\{a_1+a,\ldots,a_6+a\}$.
Note that the function $\eta_f[2[a_6]-2[0],y]$ is constant according to its divisor; and by definition we have $\eta_f[2[a_6]-2[0],y](y)=1$.
But this function must be equal to $0$ at the closed subvariety $W_{-o}$ of $J_{\mathcal{C}}$ 
and in particular at the six $2$-torsion points $a_1$, $\ldots$, $a_6$ to be coherent with the $(16,6)$-configuration in $\mathbb{P}^3$.

Fix $\eta_1$, \ldots, $\eta_4$ a basis of the level $2$  functions. This defines a map $\phi$ from $J_{\mathcal{C}}$ to the Kummer surface of
the Jacobian $J_{\mathcal{D}}$ of $\mathcal{D}$
seen in $\mathbb{P}^3$ and we
denote by $Z_1$, \ldots, $Z_4$ the
projective coordinates associated to this basis, as already done in the previous section.
For all $a$, write $\eta_f[2[a]-2[0],y]=c_1\eta_1+\ldots+c_4\eta_4$ for $c_k\in K$. Denote then $Z_{a}=c_1Z_1+\ldots+c_4Z_4$ the tropes.
The nodes are the image by $\phi$ of the $2$-torsion points. The $(16,6)$-configuration (of the Kummer surface of $\mathcal{D}$) is described by Table \ref{table:confG2}, where
for each trope we have written the $6$ $2$-torsion points whose images lie in it.

\begin{table}[h]
  \begin{tabular}{|c|cccccc|c|c|cccccc|}\hline
    $Z_{a_1}$  & $a_1$& $a_6$& $a_{12}$& $a_{13}$& $a_{14}$& $a_{15}$& $\qquad$ &$Z_{a_{14}}$& $a_1$& $a_4$& $a_{14}$& $a_{23}$& $a_{25}$& $a_{35}$\\
    $Z_{a_2}$  & $a_2$& $a_6$& $a_{12}$& $a_{23}$& $a_{24}$& $a_{25}$&          &$Z_{a_{15}}$& $a_1$& $a_5$& $a_{15}$& $a_{23}$& $a_{24}$& $a_{34}$\\
    $Z_{a_3}$  & $a_3$& $a_6$& $a_{13}$& $a_{23}$& $a_{34}$& $a_{35}$&          &$Z_{a_{23}}$& $a_2$& $a_3$& $a_{14}$& $a_{15}$& $a_{23}$& $a_{45}$\\
    $Z_{a_4}$  & $a_4$& $a_6$& $a_{14}$& $a_{24}$& $a_{34}$& $a_{45}$&          &$Z_{a_{24}}$& $a_2$& $a_4$& $a_{13}$& $a_{15}$& $a_{24}$& $a_{35}$\\
    $Z_{a_5}$  & $a_5$& $a_6$& $a_{15}$& $a_{25}$& $a_{35}$& $a_{45}$&          &$Z_{a_{25}}$& $a_2$& $a_5$& $a_{13}$& $a_{14}$& $a_{25}$& $a_{34}$\\
    $Z_{a_6}$  & $a_1$& $a_2$& $a_{3}$ & $a_{4}$ & $a_{5}$& $a_{6}$ &          &$Z_{a_{34}}$& $a_3$& $a_4$& $a_{12}$& $a_{15}$& $a_{25}$& $a_{34}$\\
    $Z_{a_{12}}$& $a_1$& $a_2$& $a_{12}$& $a_{34}$& $a_{35}$& $a_{45}$&          &$Z_{a_{35}}$& $a_3$& $a_5$& $a_{12}$& $a_{14}$& $a_{24}$& $a_{35}$\\
    $Z_{a_{13}}$& $a_1$& $a_3$& $a_{13}$& $a_{24}$& $a_{25}$& $a_{45}$&          &$Z_{a_{45}}$& $a_4$& $a_5$& $a_{12}$& $a_{13}$& $a_{23}$& $a_{45}$\\
    \hline
  \end{tabular}
        \caption{$(16,6)$-configuration in genus $2$}\label{table:confG2}
\end{table}

\begin{remm}\label{rem:sameconf}
As $\vartheta=0$, then by Equation (\ref{eq:defetax}) we have for $a,a'\in J_{\mathcal{C}}[2]$ that
$\eta[2[a]-2[0],y](a')=0$ implies that $\eta_f[2[a]-2[0],y](a')=0$.  
The converse is also true because of the $(16,6)$-configuration of the Kummer surfaces of $\mathcal{C}$ and $\mathcal{D}$.
So the description of the two configurations with the $2$-torsion points of $J_{\mathcal{C}}$ is the same.
\end{remm}

\paragraph{Image of the points in a fixed trope.}
To compute the equation of $\mathcal{D}$, we need the image in $\mathbb{P}^3$ of the $6$ $2$-torsion points of a given fixed trope.

\begin{remm}
We do not compute the $\eta_f$ functions at $2$-torsion points directly. Indeed, these functions have poles at some $2$-torsion points and the algorithm
to evaluate the $\eta_f$ functions at a $2$-torsion point does not behave well in practice (it can return an error)  because of the fact that the zero cycles used here are
also defined with $2$-torsion points (see Step $1$ in Section \ref{subsec:evaleta}).
We circumvent this difficulty using the configuration.
\end{remm}

Let $Z_a$ be a trope. It contains the six points $\phi(a_1+a),\ldots,\phi(a_6+a)$ in $\mathbb{P}^3$.
We compute the image by $\phi$ of each of these six points by computing the intersection
of $3$ tropes. Indeed, as the Kummer surface is seen in a projective space of dimension $3$, we need $3$ equations to determine a point.
We use the following properties, for $i\in\{1,\ldots,5\}$, that can be deduced from Table \ref{table:confG2}:
\[\{Z_{a_1}=0,Z_{a_2}=0,Z_{a_3}=0\}=\{\phi(a_6)\},\qquad \{Z_{a_i}=0,Z_{a_6}=0\}=\{\phi(a_i),\phi(a_6)\}.\]
Thus, by shifting, looking at the intersection  $\{Z_{a_1+a}=0,Z_{a_2+a}=0,Z_{a_3+a}=0\}$ gives us the projective point
$p_6=\phi(a_6+a)$.
Then the point $p_i=\phi(a_i+a)$ can be computed similarly with the intersection  $\{Z_{a_i+a}=0,Z_{a_6+a}=0,Z_b=0\}$ for some good choosen trope $Z_b$.
Otherwise, if we have the equation of the quartic $\kappa_{\mathcal{D}}$ describing the Kummer surface associated to $\mathcal{D}$, we can obtain this point with
$\{Z_{a_i+a}=0,Z_{a_6+a}=0,\kappa_{\mathcal{D}}=0\}$.

\paragraph{Parameterization.}
We now have a fixed trope $Z_a=c_1Z_1+\ldots+c_4Z_4=0$ and all the $6$ nodes $p_1$, \ldots, $p_6$ in it.
The intersection of a trope with the Kummer surface is a conic and this intersection is of multiplicity $2$. So $p_1$, \ldots, $p_6$ lie in a conic (in $\mathbb{P}^3$).
The double cover of the conic ramified at these points is the curve $\mathcal{D}$.
See \cite[Theorem 1.1]{DolgLeh} (or \cite[Proposition 10.2.3 and Corollary 10.2.4]{Birk} for the case $K=\mathbb{C}$).
We do a parameterization of the conic and apply it to the ramified points to obtain the Weierstrass points of $\mathcal{D}$.

\begin{enumerate}
\item Choose any point among $\{p_1,\ldots,p_6\}$, say $p_1$.
\item Let $k=\min_{i\in\{1,2,3,4\}}\{i:c_i\ne 0\}$  so that  $c_kZ_k=-\sum_{i=k+1}^4c_iZ_i$.
We look for equations of the form $E_j=c_{j,1}Z_1+\ldots+c_{j,4}Z_4$ passing through $p_1$ 
and $p_j$, for $j\in\{2,\ldots,6\}$ with $c_{j,k}=0$. Note that these lines correspond to the intersection of $Z_a$ with the unique other trope
passing through $p_1$ and $p_j$ (see Corollary \ref{cor:inter2}).
\item The fact that $E_j$ evaluated at $p_1$ must be equal to $0$ yields a relation between $c_{j,1},\ldots,c_{j,4}$.
Assume to simplify that $k=1$ so that we have $c_{j,1}=0$ and that $c_{j,2}$
can be written as a linear combination of $c_{j,3}$ and $c_{j,4}$; i.e. $c_{j,2}=P(c_{j,3},c_{j,4})$.
\item We obtain an affine parameterization in taking $c_{j,3}=1$ and $c_{j,4}=x$ and we look at the equation $E=P(1,x)Z_2+Z_3+xZ_4$.
\item For $j\in\{2,\ldots,6\}$, evaluate $E$ at $p_j$. This yields an equation of degree $0$ or $1$ in $x$.
 If it is $1$, then we obtain the value $x$ and if it is $0$, then $x$ is the point at infinity.
Thus, we have $5$ of the $6$ Weierstrass points of a model of $\mathcal{D}$.
\item The last one, associated to $p_1$, can be obtained in intersecting the equation $\kappa_{\mathcal{D}}$ of the Kummer surface with the trope $Z_{a}$
  and the equation $E$, factorizing, evaluating at $p_1$ and solving the factor having $x$ (this idea is implicit in \cite{CouvEz}).

If we want to avoid the computation of $\kappa_{\mathcal{D}}$, which is costly in terms of the number of evaluations
of the $\eta_f$ functions, a solution consists to do the parameterization two times for two different fixed points.
This yields two sets
of $5$ Weierstrass points (for different models of the curve) and we look then for a change of variables sending exactly $4$ elements of the first set in the second set. Applying the transformation on the fifth point of the first set gives us the unknown Weierstrass point of the second set. 
Recall that a nonsingular projective model of a genus $2$ curve is $Y^2=\prod_{i=1}^6c_iX^iZ^{6-i}$ in the projective space with weight $(1,3,1)$ and
that a transformation is of the form $(X:Y:Z)\mapsto (\alpha X+\beta Z:\gamma Y:\delta X+\epsilon Z)$ with $\alpha\epsilon-\beta\delta=1$.

\end{enumerate}  
\begin{remm}  Following this algorithm, we obtain a model of $\mathcal{D}$. But over a finite field, we may want to distinguish this curve from
  its twist.
    The knowledge of the cardinality
    of the Jacobian of $\mathcal{C}$ is a sufficient data for it.
  \end{remm}

\subsection{Optimized algorithm for the computation of $\mathcal{D}$ in genus $2$} 
The method exposed above with the computation of $\kappa_{\mathcal{D}}$ as in \cite{CouvEz} requires to compute the equation of the Kummer surface and the $6$ tropes $Z_{a_i+a}$, for $i\in\{1,\ldots,6\}$.
This makes at most $35\times 4+(6+4)\times 4=180$ evaluations of $\eta_f$ functions.
We have already showed that this number is greatly reduced as the computation of the equation of the Kummer surface is not necessary.
We explain how to optimize the computation of the image of the $2$-torsion points of a well chosen trope.

To simplify the notations, we denote by $\eta_a$ the function $\eta_f[2[a]-2[0],y]$.
A good choice of basis is $\eta_1=\eta_{a_6},\eta_2=\eta_{a_1},\eta_3=\eta_{a_2},\eta_4=\eta_{a_{12}}$
($Z_{a_6}$ and $Z_{a_1}$ contain $\phi(a_1)$ while $Z_{a_2}$ not, and $Z_{a_6}$, $Z_{a_1}$, $Z_{a_2}$ contain $\phi(a_6)$ while $Z_{a_{12}}$ not; this proves that the four functions are independent). 

\begin{remm}
We have noted that, with this basis, the equation of the Kummer surface $\kappa_{\mathcal{D}}$ does not have any exponent of
degree $3$ and $4$. So it has at most $19$ coefficients and the cost of the computation of $\kappa_{\mathcal{D}}$ is reduced.
\end{remm}

The fixed trope we choose is $Z_{a_6}$. This implies we want to compute the points $\phi(a_i)$ for $i\in\{1,\ldots,6\}$.
As $Z_{a_6}$ contains the $6$ points $\phi(a_i)$, we already know that the first coordinates in $\mathbb{P}^3$ of these $6$ points is $0$.
We have (according to Table~\ref{table:confG2})
\[p_6=\phi(a_6)=(0:0:0:1),\quad p_1=\phi(a_1)=(0:0:1:0),\quad p_2=\phi(a_2)=(0:1:0:0).\]
Fixing the point $p_6$, we obtain
the affine parameterization $Z_1=0$, $Z_2+xZ_3=0$. Then $p_1$ comes from $x=0$ while $p_2$ from $x=\infty$.
Thus, with this basis, we always obtain a degree $5$ model for
the isogenous curve $\mathcal{D}$ and the image of three of the six points in the fixed trope $Z_{a_6}$ are obtained for free.
 It remains to compute the images of $a_3$, $a_4$ and $a_5$ in $\mathbb{P}^3$. We have
\begin{itemize}
\item $\{Z_{a_6}=0,Z_{a_{34}}=0,Z_{a_{35}}=0\}$ gives us $p_3=\phi(a_3)$,
\item $\{Z_{a_6}=0,Z_{a_{34}}=0,Z_{a_{45}}=0\}$ gives us $p_4=\phi(a_4)$,
\item $\{Z_{a_6}=0,Z_{a_{35}}=0,Z_{a_{45}}=0\}$ gives us $p_5=\phi(a_5)$.
\end{itemize}
As by definition $Z_{a_6}=Z_1$, $Z_{a_1}=Z_2$, $Z_{a_2}=Z_3$ and $Z_{a_{12}}=Z_4$, we have to compute only the $3$ tropes $Z_{a_{34}}$, $Z_{a_{35}}$ and $Z_{a_{45}}$.
The computation is simplified for two reasons.
\begin{itemize}
\item 
For any $a\in J_{\mathcal{C}}[2]$, there is $a'\in\{a_6,a_1,a_2,a_{12}\}$ such that $Z_{a}$ contains $\phi(a')$ and such that exactly three of the four tropes
$Z_1$, \ldots, $Z_4$  contain $\phi(a')$.
So all the tropes can be written as a linear combination of three elements
among $\{Z_1,Z_2,Z_3,Z_4\}$ and thus we need the evaluation at $3$ points for computing a trope.
In particular, $Z_{a_{34}}$, $Z_{a_{35}}$ and $Z_{a_{45}}$ are linear combinations of $Z_2$, $Z_3$ and $Z_4$.
\item As we have defined all
the $\eta_f$ functions such that their values at $y$ is $1$, we just need the evaluations at two more  points.
\end{itemize}

Thus, we can obtain the images of $a_1,\ldots,a_6$ in $\mathbb{P}^3$ by computing the tropes $Z_{a_{34}}$, $Z_{a_{35}}$ and $Z_{a_{45}}$ which
can be done in $(3+3)\times 2=12$ evaluations of $\eta_f$ functions.

There is a slight amelioration in noting that for $a,b\in J_{\mathcal{C}}[2]$ and $x\in J_{\mathcal{C}}$
\begin{equation}\eta_{a+b}(x)=\eta_f[2[a]-2[0],y+b](x+b)\cdot\eta_b(x)\end{equation}
(look at the divisors and the evaluations at $y$ for the proof). Let $b=a_{45}$, and take some random point $z\in J_{\mathcal{C}}$ (not of $2$-torsion).
The function $\eta_{a_{35}}(x+a_{45})\cdot\eta_{a_{45}}(x)$ has the same divisor as $\eta_{a_{34}}(x)$.
So there is a constant $c$ such that $\eta_{a_{34}}(x)=c\cdot \eta_{a_{35}}(x+a_{45})\cdot\eta_{a_{45}}(x)$. Evaluating at $y$ we obtain
$1=c\cdot \eta_{a_{35}}(y+a_{45})$.
It remains to evaluate $\eta_{a_i}$ for $a_i\in\{a_1,a_2,a_{12},a_{35},a_{45}\}$ at the points $z$ and $z+a_{45}$ for computing $Z_{a_{34}}$, $Z_{a_{35}}$, $Z_{a_{45}}$.
So $1+5\times 2=11$ evaluations are enough instead of $12$.

We obtain Algorithm \ref{alg:eval2points} whose complexity is the same as the one for evaluating $\eta_f$ at a point but whose number of calls to a $\eta_f$ function
is minimized. \newline

\begin{algorithm}[H]
\KwData{The basis $\eta_{a_6}$, $\eta_{a_1}$, $\eta_{a_2}$, $\eta_{a_{12}}$ of the level $2$ functions}
\KwResult{$\phi(a_1),\ldots,\phi(a_6)$ in $\mathbb{P}^3$, where $\phi=(\eta_{a_6}:\eta_{a_1}:\eta_{a_2}:\eta_{a_{12}})$}
\BlankLine
\nl Take a random point in $z\in J_{\mathcal{C}}$\;
\nl Evaluate $\eta_{a_i}$ at $z$ and $z+a_{45}$, for $a_i\in\{a_1,a_2,a_{12},a_{35},a_{45}\}$\;
\nl Compute $c=1/\eta_{a_{35}}(y+a_{45})$\;
\nl Compute $c\cdot \eta_{a_{35}}(z+a_{45})\cdot \eta_{a_{45}}(z)$ and $c\cdot \eta_{a_{35}}(z)\cdot \eta_{a_{45}}(z+a_{45})$ using the previous evaluations.
These are $\eta_{a_{34}}(z)$ and $\eta_{a_{34}}(z+a_{45})$\;
\nl Compute the tropes $Z_{a_{34}}$, $Z_{a_{35}}$ and $Z_{a_{45}}$ which are of the form $c_2 Z_2+c_3 Z_3+c_4 Z_4$ using the previous
evaluations and the fact that $1=c_2+c_3+c_4$ (evaluation at $y$)\;
\nl We know that $\phi(a_6)=(0:0:0:1)$, $\phi(a_1)=(0:0:1:0)$ and that $\phi(a_2)=(0:1:0:0)$\;
\nl $\{Z_{a_6}=0,Z_{a_{34}}=0,Z_{a_{35}}=0\}$ gives $\phi(a_3)$\;
\nl $\{Z_{a_6}=0,Z_{a_{34}}=0,Z_{a_{45}}=0\}$ gives $\phi(a_4)$\;
\nl $\{Z_{a_6}=0,Z_{a_{35}}=0,Z_{a_{45}}=0\}$ gives $\phi(a_5)$\;
\caption{Computation of $\phi(a_1),\ldots,\phi(a_6)$ with $11$ evaluations}
\label{alg:eval2points}
\end{algorithm}

\paragraph{Example.}
Let $\mathcal{C}$ be given by the equation $Y^2=(X-179)(X-237)(X-325)(X-344)(X-673)$ over $\mathbb{F}_{1009}$.
A maximal isotropic subgroup of the $\ell=3$ torsion is generated by the divisors (in Mumford representation)
$T_1=\langle X^2 + 714X + 513, 182X + 273\rangle$ and  $T_2=\langle X^2 + 654X + 51, 804X + 545\rangle$.
We fix $y=\langle X^2 + 425X + 637, 498X + 930\rangle$, $\phi_u=\langle X^2 + 462X + 658, 365X + 522\rangle$,  
$\phi_y=\langle X^2 + 512X + 883, 827X + 148\rangle$.
We put $r_1=(179,0)$, $r_2=(237,0)$, $r_3=(325,0)$, $r_4=(344,0)$, $r_5=(673,0)$ and $r_6=\infty$.
We take the good basis of $\eta_f$ functions defined by the zero-cycles $2[a]-2[0]$ for $a\in\{0,r_1-r_6,r_2-r_6,r_1+r_2-2r_6\}$. Then
\[ Z_{a_6}=Z_1,\qquad Z_{a_1}=Z_2,\qquad Z_{a_2}=Z_3,\qquad Z_{a_{34}}= 953Z_2 + 55Z_3 + 2Z_4, \]
\[Z_{a_{35}}= 806Z_2 + 131Z_3 + 73Z_4,\qquad Z_{a_{45}}=894Z_2 + 123Z_3 + 1002Z_4\]
giving us the nodes
$(0:0:1:0)$, $(0:1:0:0)$, $(0:947:689:1)$, $(0:304:71:1)$, $(0:869:468:1)$, $(0:0:0:1)$
which are in the trope $Z_1=0$. Then
\begin{itemize}
\item 
Fixing the point $(0:0:0:1)$, we take the parameterization  $Z_1=0$, $Z_2+xZ_3=0$ and we obtain the
values  $\{0,\infty,498,351,397\}$ for $x$ respectively.
\item 
Fixing the point $(0:0:1:0)$, we take the parameterization  $Z_1=0$, $Z_2+xZ_4=0$ and we obtain the
values  $\{\infty,62,705,140,0\}$ for $x$ respectively.
\end{itemize}
For the transformation, we take the one sending $498$ to $62$ and $351$ to $705$ which is $(X:Y:Z)\mapsto (229X+37Z:Y:Z)$.
Then $397$ is sent to $140$, $0$ to $37$, $\infty$ to $\infty$ and $837$ to $0$.
Two models of the curve $\mathcal{D}$ over $\mathbb{F}_{1009}$ are
$11X(X-498)(X-351)(X-397)(X-837)$ and $X(X-62)(X-705)(X-140)(X-37)$ (after checking quadratic twist).

\subsection{Computing the equation of the isogenous curve in genus 3 if $\mathcal{D}$ is hyperelliptic}

We now focus  on the genus $3$ case and we assume that $\mathcal{D}$ is hyperelliptic with an imaginary model.
We make use of the $(64,29)$-configuration to compute the equation of $\mathcal{D}$ and this configuration does not depend on
$\mathcal{C}$ so the nature of this curve does not matter in theory. But when $\mathcal{C}$ is also hyperelliptic
(with an imaginary model), the link between the two curves is clearer because the description of the $2$-torsion is similar 
and working on $\mathcal{C}$ is as if we were working directly on $\mathcal{D}$
(just replace the $\eta_f$ functions  by the $\eta$ functions on $\mathcal{D}$, see Remark \ref{rem:sameconf}).

So in our exposition we assume that $\mathcal{C}:Y^2=\prod_{i=1}^7(X-r_i)$ is hyperelliptic.
Let $\eta_1$, \ldots, $\eta_8$ be a basis of the functions $\eta_f[2[a]-2[0],y]$
for $a\in J_{\mathcal{C}}[2]$.
Recall that the $64$ $2$-torsion points are denoted by $a_i$, $a_{ij}$ and $a_{ijk}$.
The trope $Z_{a_8}$ contains the $8$ points $\phi(a_i)$ and the $21$ points $\phi(a_{ij})$.

The $(64,29)$-configuration holds properties that the $(64,28)$-configuration does not have.
\begin{itemize}
\item 
Let $i\in\{1,\ldots,7\}$. As the image of the points $\{a_{i1},\ldots,a_{i7},a_i\}$ are in the trope $Z_{a_8}$, then the trope
$Z_{a_8+a_i}=Z_{a_i}$ contains the points $\{\phi(a_{i1}+a_i),\ldots,\phi(a_{i7}+a_i),\phi(a_i+a_i)\}=$ $\{\phi(a_1),\ldots,\phi(a_8)\}$.
Thus, the intersection of the $8$ tropes $\{Z_{a_1},\ldots,Z_{a_8}\}$ is equal to $\{\phi(a_1),\ldots,\phi(a_8)\}$
and in fact, any $4$ tropes among these $8$ have this intersection. We use this to compute the set $\{\phi(a_1),\ldots,\phi(a_8)\}$.
\item For any triplet of points among $\{\phi(a_1),\ldots,\phi(a_8)\}$, there always exists a trope not in
  $\{Z_{a_1},\ldots,Z_{a_8}\}$ which contain these three points and no other among them. We use this for the parameterization step.
  \end{itemize}
These properties can be proved using Proposition \ref{prop:confG3}.
There are obviously $64$ $8$-tuples of tropes having similar properties (just shift by a $2$-torsion point).

To compute the $8$ points $\phi(a_i)$, we choose the tropes $Z_{a_{ij}}$ for $ij \in\{24,37,67\}$ and $Z_{a_{ijk}}$ for $ijk\in\{123,145,167,256,345\}$ because 
each point $\phi(a_i)$ for $i\in\{1,\ldots,8\}$ is contained in exactly three of these $8$ tropes.
For any $a_i$, this gives $3$ tropes and adding the four tropes $\{Z_{a_1},\ldots,Z_{a_4}\}$,
we obtain $7$ equations from which we deduce  $\phi(a_i)$ in $\mathbb{P}^7$.
So computing $12$ tropes is enough to obtain the image of the eight $2$-torsion points $a_1,\ldots,a_8$ in $\mathbb{P}^7$.
If the basis $\eta_1$, \ldots, $\eta_8$  is defined using $8$ of the $12$ $2$-torsion points used for these $12$ tropes, then we only need to compute $4$ tropes.

Once we have  $\{\phi(a_1),\ldots,\phi(a_8)\}$ in $\mathbb{P}^7$, we can do the parameterization.
The four tropes  $\{Z_{a_1},\ldots,Z_{a_4}\}$ give $4$ equations. This time, we fix two points instead of one, which gives us $2$ others equations.
The rest is similar as in the genus $2$ case. Instead of lines, here we have planes, according to the second property above, passing through the two fixed points and a third one.

In the case where the curve $\mathcal{C}$ is non-hyperelliptic, we can still compute all the tropes and the
image of the $2$-torsion points and look for $4$ tropes intersecting in $8$ points, and proceed as
above. 

\section{Computing equations for the isogeny}\label{sec:eqiso}
Once we have the equation of the hyperelliptic curve $\mathcal{D}$ of genus $2$ or $3$, we want to compute rational fractions expliciting the isogeny.
The algorithm is composed as follows.
\begin{itemize}
\item Choose a basis of $\eta_f$ functions. This determines equations for the Kummer variety of $\mathcal{D}$.
  Find a linear change of variables to go from this model of the Kummer variety to a good representation of it, allowing one
  to compute the pseudo-addition law and to lift to the Jacobian.
\item Choose a single point in $\mathcal{C}(K[t])$, for a formal parameter $t$.
  Compute its image in the Kummer variety of $\mathcal{D}(K[t])$ by the isogeny, using the basis of $\eta_f$ functions, at small precision in~$t$.
\item Lift the point in the Kummer variety to a point $p$ in $J_{\mathcal{D}}(K[t])$.
\item Extend the point $p$ at a big enough precision.
\item Use $p$ and continuous fractions to compute the rational fractions.
\end{itemize}

The last two steps come from \cite{CouvEz} (genus $2$ case only).
For computing the image of a point of $\mathcal{C}(K[t])$ to $J_{\mathcal{D}}(K([t]))$,
the method given in \cite{CouvEz} is not efficient so we present a better solution based on a good representation of the Kummer variety.


\begin{remm} Even if we know $\mathcal{D}$, it is not possible for now to write a $\eta_f$ function  as a combination
  of $\eta$ functions defined over $J_{\mathcal{D}}$ (and not $J_{\mathcal{C}}$)
  because in the first case we work on $J_{\mathcal{C}}$ and in the others on $J_{\mathcal{D}}$.
  If we wish to do so, we would need for $x\in J_{\mathcal{C}}$ to know the point $f(x)\in J_{\mathcal{D}}$, which is what we want to compute.
\end{remm}

\subsection{Rational fractions describing the isogeny}\label{subsec:ratfrac}
See \cite[Section 6.1]{CouvEz} for more details in genus $2$.
Let $g\in\{2,3\}$. Assume we have $\mathcal{D}$ given by an affine model $Y^2=h_{\mathcal{D}}(X)$, where $h_{\mathcal{D}}$ is of degree $2g+1$.
Let $O_{\mathcal{D}}$ be the point at infinity.
Then $(2g-2)O_{\mathcal{D}}$ is a canonical divisor
and a point in the Jacobian $J_{\mathcal{D}}$ of $\mathcal{D}$ can be written generically as $z=Q_1+\ldots+Q_g-gO_{\mathcal{D}}$,
where $Q_i\ne O_{\mathcal{D}}$ and  $-Q_i\not\in\{Q_1,\ldots,Q_g\}$ for  all $i$ in $\{1,\ldots,g\}$.
Such a divisor can be represented by its Mumford representation. 

For $g=2$, define
\[  \mathbf{s}(z)=x(Q_1)+x(Q_2),\qquad  \mathbf{p}(z)=x(Q_1)x(Q_2)\]
\[  \mathbf{q}(z)=\frac{y(Q_2)-y(Q_1)}{x(Q_2)-x(Q_1)} \qquad   \mathbf{r}(z)=\frac{y(Q_1)x(Q_2)-y(Q_2)x(Q_1)}{x(Q_2)-x(Q_1)}.\]
 The Mumford representation of $z$ is \[\langle X^2-\mathbf{s}(z)X+\mathbf{p}(z),\mathbf{q}(z)X+\mathbf{r}(z)\rangle.\]
 Let now $F:\mathcal{C}\to J_{\mathcal{D}}$ be the function $F(P)=f(\iota(P-O_{\mathcal{C}}))$ 
 (recall that $f$ is the isogeny, $\iota$ stands for linear classes of divisors and we denote here by $O_{\mathcal{C}}$ the unique point at infinity of an imaginary model of $\mathcal{C}$). Since for every point $P=(u,-v)$ on $\mathcal{C}$
 we have that $F(-P)=F((u,-v))=-F(P)$, and as $v^2=h_{\mathcal{C}}(u)$, we deduce that there exist rational fractions
 $\mathbf{S}$,  $\mathbf{P}$,  $\mathbf{Q}$,  $\mathbf{R}$
 satisfying
 \[ \mathbf{s}(F(P))=\mathbf{S}(u),\quad  \mathbf{p}(F(P))=\mathbf{P}(u),\quad  \mathbf{q}(F(P))=v\mathbf{Q}(u),\quad  \mathbf{r}(F(P))=v\mathbf{R}(u)\]
 and such that $F((u,v))=\langle X^2-\mathbf{S}(u)X+\mathbf{P}(u),v(\mathbf{Q}(u)X+\mathbf{R}(u))\rangle$
in the Jacobian $J_{\mathcal{D}}$ of $\mathcal{D}$ in Mumford representation.
The degrees of these rational fractions are bounded by $2\ell$, $2\ell$, $3\ell+3$, $3\ell+3$ respectively
(we adapt the proof of \cite[Section 6.1]{CouvEz} considering that $\mathcal{D}:Y^2=h_{\mathcal{D}}(X)$ with $h_{\mathcal{D}}$ of degree $5$).
 
For $g=3$, define
 \[\begin{array}{ccl}
    \mathbf{s}(z)&=&\text{\scriptsize $x(Q_1)+x(Q_2)+x(Q_3)$},\\
    \mathbf{p}(z)&=&\text{\scriptsize $x(Q_1)x(Q_2)+x(Q_1)x(Q_3)+x(Q_2)x(Q_3)$},\\
    \mathbf{a}(z)&=&\text{\scriptsize $x(Q_1)x(Q_2)x(Q_3)$},\\
    \mathbf{r}(z)&=&\frac{((x(Q_2)-x(Q_3))y(Q_1)+(x(Q_3)-x(Q_1))y(Q_2)+(x(Q_1)-x(Q_2))y(Q_3))}{(x(Q_1)-x(Q_2))(x(Q_1)-x(Q_3))(x(Q_2)-x(Q_3))},\\
    \mathbf{t}(z)&=&\frac{(X^2(Q_2)-X^2(Q_3))y(Q_1)+(X^2(Q_3)-X^2(Q_1))y(Q_2)+(X^2(Q_1)-X^2(Q_2))y(Q_3)}{((x(Q_1)-x(Q_2))(x(Q_1)-x(Q_3))(x(Q_2)-x(Q_3))},\\
    \mathbf{e}(z)&=&\frac{(X^2(Q_2)x(Q_3)-x(Q_2)X^2(Q_3))y(Q_1)+(x(Q_1)X^2(Q_3)-X^2(Q_1)x(Q_3))y(Q_2)+(X^2(Q_1)x(Q_2)-x(Q_1)X^2(Q_2))y(Q_3)}{(x(Q_1)-x(Q_2))(x(Q_1)-x(Q_3))(x(Q_2)-x(Q_3))},\\
  \end{array}
 \] 
 so that the Mumford representation of $z$ is 
\[\langle X^3-\mathbf{s}(z)X^2+\mathbf{p}(z)X-\mathbf{a}(z),\mathbf{r}(z)X^2-\mathbf{t}(z)X+\mathbf{e}(z) \rangle.\]
By the same argument from genus $2$, there exist rational fractions
 $\mathbf{S}$,  $\mathbf{P}$,  $\mathbf{A}$,  $\mathbf{R}$,  $\mathbf{T}$,  $\mathbf{E}$ satisfying
 \[ \mathbf{s}(F(P))=\mathbf{S}(u),\quad
  \mathbf{p}(F(P))=\mathbf{P}(u),\quad
  \mathbf{a}(F(P))=\mathbf{A}(u),\]\[
  \mathbf{r}(F(P))=v\mathbf{R}(u),\quad
  \mathbf{t}(F(P))=v\mathbf{T}(u),\quad
  \mathbf{e}(F(P))=v\mathbf{E}(u),
\]
and such that  $F((u,v))=\langle X^3-\mathbf{S}(u)X^2+\mathbf{P}(u)X-\mathbf{A}(u) ,v(\mathbf{R}(u)X^2-\mathbf{T}(u)X+\mathbf{E}(u))\rangle$.

\subsection{Computing the rational fractions from the image of a single formal point}
See \cite[Section 6.2]{CouvEz} for more details in genus $2$. Again $g\in\{2,3\}$.
The morphism $F:\mathcal{C}\to J_{\mathcal{D}}$ induces a map $F^*:H^0(J_{\mathcal{D}},\Omega^1_{J_{\mathcal{D}}/K})\to H^0(\mathcal{C},\Omega^1_{\mathcal{C}/K})$.
It is a classical result that a basis of $H^0(\mathcal{C},\Omega^1_{\mathcal{C}/K})$ is given by $dX/Y$, \ldots, $X^{g-1}dX/Y$.
Identifying $J_{\mathcal{D}}$ with $\mathcal{D}^{(g)}$ (the symmetric product)
we can see $H^0(J_{\mathcal{D}},\Omega^1_{J_{\mathcal{D}}/K})$ as the invariant subspace
of $H^0(\mathcal{D}^{(g)},\Omega^1_{\mathcal{D}^{(g)}/K})$ by the permutation of $g$ factors. A basis of this space is
$dX_1/Y_1+\ldots+dX_g/Y_g$, \ldots, $X_1^{g-1}dY_1/Y_1+\ldots+X_g^{g-1}dX_g/Y_g$. Let $(m_{i,j})_{1\le i,j\le g}$ be the matrix of $F^*$ with respect to these two bases. Thus for $i\in\{1,\ldots,g\}$
\[F^*(X_1^{i-1}dX_1/Y_1+\ldots+X_g^{i-1}dX_g/Y_g)=(m_{1,i}+\ldots+m_{g,i}X^{g-1})dX/Y.\]

Let $P=(u,v)$ be a point on $\mathcal{C}$ such that $v\ne 0$ and let $Q_i$ be $g$ points on $\mathcal{D}$  as in the previous subsection,
such that $F(P)$ is the class of $Q_1+\ldots+Q_g-gO_{\mathcal{D}}$.
Let $t$ be a formal parameter and set $L=K[t]$. Define $u(t)=u+t$ and $v(t)$ as the square root of $h_{\mathcal{C}}(u(t))$
which is equal to $v$ when $t=0$. The point $P(t)=(u(t),v(t))$ lies on $\mathcal{C}(L)$. The image of $P(t)$ by $F$ is the class
of $Q_1(t)+\ldots+Q_g(t)-gO_{\mathcal{D}}$ for $g$ $L$-points $Q_1(t)$, \ldots, $Q_g(t)$ on $\mathcal{D}(L)$. We explain in the next subsection how to compute them
at a given precision.
Write $Q_i(t)=(x_i(t),y_i(t))$. The coordinates satisfy the non-singular first-order system of differential equations
for $i\in\{1,\ldots,g\}$
\begin{equation}\label{eq:system}\left\{
    \begin{array}{l}
      \frac{x_1^{i-1}\dot{x}_1(t)}{y_1(t)}+\ldots+\frac{x_g^{i-1}\dot{x}_g(t)}{y_g(t)}=\frac{(m_{1,i}u(t)^0+\ldots+m_{g,i}u(t)^{g-1})\dot{u}(t)}{v(t)},\\
      {\text{\scriptsize $y_i(t)^2=h_{\mathcal{D}}(x_i(t))$}}.\\
    \end{array}\right.
\end{equation}

This system can be used to compute the rational fractions of Section \ref{subsec:ratfrac} in three steps.
Indeed, assume we have been able to compute for a single point $(u(t),v(t))$ the $g$ points $(x_j(t)+O(t^g),y_j(t)+O(t^g))$ at precision $g$.
\begin{enumerate}
\item 
  Looking at coefficients of degrees from $0$ to $g-1$ in the first line of Equation (\ref{eq:system}) for a fixed index $i$
  gives $g$ equations with the $g$ unknown  
  $m_{1,i}$, \ldots, $m_{g,i}$ that we can solve. Thus we obtain the numbers $m_{j,i}$ for $i,j\in\{1,\ldots,g\}$.
\item Now, we want to increase the accuracy of the formal expansions. This can be done degree by degree. The RHS of the first line of Equation 
  (\ref{eq:system})
  is known up to any given precision. Assume we know $x_j(t)$ and
  $y_j(t)$ up to $O(t^d)$ for all $j$ (and their derivatives up to $O(t^{d-1})$).
  If $c_{j,d}$ is the coefficient of degree $d$ of $x_j(t)$, then the coefficient of degree $d-1$ of its derivative is $d\cdot c_{j,d}$.
  For $j\in\{1,\ldots,g\}$, define $\dot{x}_j^{d-1}(t)$ as 
  the sum of $\dot{x}_j(t)$ up to degree $d-2$ in $t$ with $d\cdot c_{j,d}t^{d-1}$,
  where $c_{j,d}$ is a variable. Plug it in Equation  (\ref{eq:system}) and deduce
  for each $i$ an equation in the $c_{j,d}$ looking
  at the coefficients of degree $d-1$ in $t$. This gives $g$ equations with $g$ unknown that we solve.
  The second line of Equation (\ref{eq:system}) allows us to compute $y_1(t)+O(t^{d+1})$ and $y_2(t)+O(t^{d+1})$.
\item Do rational reconstruction using continued fractions to deduce the rational fractions. For example, for  $\textbf{S}$ in genus $2$, put $s(t)=x_1(t)+x_2(t)$
  and remark that $s(t)=s_{\le 0}(t)+1/(1/(s(t)-s_{\le 0}(t)))$, where $s_{\le 0}(t)$ designates the sum of the monomials of degree less or equal to $0$ in $s(t)$.
  So while the degree of $s$ in $t$ is $>0$, store $s_{\le 0}(t)$ as a rational fraction in $t$ in a stack and put $s(t)=1/(s(t)-s_{\le 0}(t))\in K((t))$.
  After the while loop, for each element $s_p$ of the stack do $s(t)=1/s(t)+s_p$.
  This gives a rational fraction in $t$.
  Evaluate it in $t-u(0)$ to obtain $\textbf{S}(t)$.  
\end{enumerate}
The complexity of these steps is independent of $\ell$.
In practice, these three steps are negligible with respect to the time of computation of the image of the single formal point.

\subsection{Computing the image of a formal point}
Let $L=K[t]/(t^g)$ and  $P(t)=(u(t),v(t))\in \mathcal{C}(L)$. We want to compute $F(P(t))$ which is in the Jacobian
of $\mathcal{D}$ over the field $L$.

\begin{remm}
We could do it using intersection of tropes. This implies we have to look at the divisor $\mathcal{Y}_{F(P(t))}$ (which is not symmetric),
seen as $\mathcal{X}_{P(t)-O_{\mathcal{C}}}$ (see Subsection \ref{subsec:etaf} for the definition of $\mathcal{X}$ and $\mathcal{Y}$).
To obtain this divisor, we could consider the zero-cycle
$[P(t)-O_{\mathcal{C}}]+[Q]+[-(P(t)-O_{\mathcal{C}})-Q]-3[0]$ for some point $Q\in J_{\mathcal{C}}$, producing
a function in $H^0(J_L,\mathcal{O}_{J_L}(3\mathcal{X}))$, that is a level $3$ function. Thus, we would need a basis of level $3$ functions and
algebraic relations between them, which is costly to compute. This is the idea in \cite[Section 6.3]{CouvEz}. 
\end{remm}

We propose to compute $F(P(t))$  in two steps. First we compute the image of $P(t)$ in the Kummer surface of $\mathcal{D}$ and then we lift
this point to the Jacobian.
The lifting step is easy to do in genus $2$ if the Kummer surface is constructed as in \cite{Prolegomena} or as in \cite{Stubbs,Muller,Stoll} in genus $3$.
Thus for any given representation of the Kummer variety, we can search for
a linear change of variables allowing one to transform to the one in good representation. We recall first what these two good representations are.

\subsubsection{Representation of the Kummer variety}
\paragraph{Standard representation of the Kummer surface.}
Let $\mathcal{D}:Y^2=h_{\mathcal{D}}(X)=\sum_{i=0}^{5}c_iX^i$ be a hyperelliptic curve
and $x=(x_1,y_1)+(x_2,y_2)-2O_{\mathcal{D}}$ a generic reduced divisor.
Let $F_0(x_1,x_2)=2c_0+c_1(x_1+x_2)+2c_2(x_1x_2)+c_3(x_1+x_2)x_1x_2+2c_4(x_1x_2)^2+c_5(x_1+x_2)(x_1x_2)^2$ and 
$\beta_0(x)=(F_0(x_1,x_2)-2y_1y_2)/(x_1-x_2)^2$.
Put
\[K_2=e_2^2-4e_1e_3,\qquad K_1=-2(2c_0e_1^3+c_1e_1^2e_2+2c_2e_1^2e_3+c_3e_1e_2e_3+2c_4e_1e_3^2+c_5e_2e_3^2),\]
\[ K_0=(c_1^2-4c_0c_2)e_1^4-4c_0c_3e_1^3e_2-2c_1c_3e_1^3e_3-4c_0c_4e_1^2e_2^2+4(c_0c_5-c_1c_4)e_1^2e_2e_3+\]\[(c_3^2+2c_1c_5-4c_2c_4)e_1^2e_3^2-4c_0c_5e_1e_2^3-4c_1c_5e_1e_2^2e_3-4c_2c_5e_1e_2e_3^2-2c_3c_5e_1e_3^3-c_5^2e_3^4.\]
Then an equation for the Kummer surface of the Jacobian of $\mathcal{D}$ is $\kappa_{\mathcal{D},opt}^{(2)}:K_2e_4^2+K_1e_4+K_0$ in the variables $e_1$, $e_2$, $e_3$, $e_4$.

The image of the divisor $x$ is $(1:x_1+x_2:x_1x_2:\beta_0(x))$ in the Kummer surface associated to $\mathcal{D}$, seen in $\mathbb{P}^3$, and represented by the equation
$\kappa_{\mathcal{D},opt}^{(2)}$.
In the case where the divisor $x$ is of the form $(x_1,y_1)-O_{\mathcal{D}}$, its image is $(0:1:x_1:c_5x_1^2)$ and the image of $0\in J_{\mathcal{D}}$ is
$(0:0:0:1)$.
Thus, if we have a point in the Kummer surface represented in this way, and using the equation of $\mathcal{D}$, it is easy to deduce the two corresponding opposite points in the Jacobian.

\begin{remm}
  We want to compute equations giving the image in $J_{\mathcal{D}}$ by $f$ of a point of $J_{\mathcal{C}}$.
  But $f$ is not completely determined by the kernel $\mathcal{V}$ and the curve $\mathcal{D}$. Indeed, these two data do not allow us to
  distinguish $f$ and $-f$. In our algorithm, we need to compute the image of only one formal point. We can choose randomly between the two opposite points in the Jacobian and this will determine our isogeny. Note that if we had to compute the image of many points for determining the isogeny, there could be a problem of compatibility between the random choices.
\end{remm}

\paragraph{Representation of the Kummer variety of dimension $3$.}
The preceding representation has been generalized in genus $3$ in \cite[Chapter 3]{Stubbs}.
The author defines for a genus $3$ hyperelliptic curve of the form $Y^2=h_{\mathcal{D}}(X)$, with $h_{\mathcal{D}}(X)$ of degree $7$,
eight functions defining a map from the Jacobian to the Kummer variety, seen in $\mathbb{P}^7$.
In particular, for a generic reduced divisor $x=(x_1,y_1)+(x_2,y_2)+(x_3,y_3)-3O_{\mathcal{D}}$, the four first functions are
$1$, $x_1+x_2+x_3$, $x_1x_2+x_1x_3+x_2x_3$, $x_1x_2x_3$ so that lifting to the Jacobian is easy.
We do not write here all the equations but refer the reader to \cite[Section 2]{Muller}, where the author extends the embedding to the Kummer variety
to non-generic divisors.
On the other side, the author of \cite{Stoll} has defined eight functions $\xi_1$, \ldots, $\xi_8$ on an arbitrary hyperelliptic curve
of genus $3$ (over a field of characteristic different from $2$) defining an embedding to the Kummer variety (see \cite[Section 3]{Stoll}).
These functions are build from the $8$ functions of \cite{Stubbs,Muller}.
These are the functions we use and we denote by $\kappa_{\mathcal{D},opt}^{(3)}$ the associated set of equations for the Kummer variety,
which is described by $1$ quadric and $34$ quartics. The lift to the Jacobian is also easy (the four first functions are the same as the four mentioned above, up to a constant) and the pseudo-addition law is described.

\subsubsection{From a representation to another.}
We now  explain  how to change representation.
Recall that we started from the curve $\mathcal{C}$ and through a basis $\eta_1$, \ldots, $\eta_{2^g}$ of $\eta_f$ functions, which are level $2$  functions,
we can obtain by linear algebra equations $\kappa_{\mathcal{D}}$ of the Kummer variety of the Jacobian of $\mathcal{D}$ and 
the image of all the $2$-torsion points in it. On the other side, starting from the equation of $\mathcal{D}$, the computation of the equations
$\kappa_{\mathcal{D},opt}^{(2)}$ or $\kappa_{\mathcal{D},opt}^{(3)}$ and the image of all the $2$-torsion points in them is easy and does not depend on the parameter $\ell$.

\paragraph{Genus $2$ case.}
We look for a change of variables to go from the quartic $\kappa_{\mathcal{D}}$ to the quartic $\kappa_{\mathcal{D},opt}^{(2)}$ of the form
\[\begin{array}{cclccl}
    S_1&=&m_1Z_1+m_2Z_2+m_3Z_3+m_4Z_4,\quad&    S_2&=&m_5Z_1+m_6Z_2+m_7Z_3+m_8Z_4,\\
    S_3&=&m_{9}Z_1+m_{10}Z_2+m_{11}Z_3+m_{12}Z_4,\quad&    S_4&=&m_{13}Z_1+m_{14}Z_2+m_{15}Z_3+m_{16}Z_4,
\end{array}\]
such that
\begin{equation}\label{eq:kummer}\kappa_{\mathcal{D},opt}^{(2)}(S_1,S_2,S_3,S_4)=\kappa_{\mathcal{D}}(Z_1,Z_2,Z_3,Z_4).\end{equation}

We give two solutions for obtaining the change of variables.
\begin{enumerate}  
\item Comparing the coefficients in the variables $Z_1$, \ldots, $Z_4$ in Equation (\ref{eq:kummer}) gives us many conditions on the $m_i$.
In theory, we can do a Gröbner basis with the $16$ unknown $m_1,\ldots,m_{16}$ to find a solution of this system of equations. But in practice, in all the small examples we tried,
the computations were so long that we stopped them before their end.
  
We can add some conditions to facilitate the Gröbner basis computation noting that we can send $0=(0:0:0:1)\in\kappa_{\mathcal{D}}$ to $0=(0:0:0:1)\in\kappa_{\mathcal{D},opt}^{(2)}$. This gives $m_4=m_8 =m_{12}=0$.
We can not put $m_{16}=1$ despite projectivity because of the equality we want between the quartics representing the Kummer surfaces.

This solution requires to compute the equation $\kappa_{\mathcal{D}}$, which is costly.
\item 
  In the case we do not want to compute $\kappa_{\mathcal{D}}$, we can look for a transformation sending the $2$-torsion points in $\kappa_{\mathcal{D}}$ to
  the $2$-torsion points in $\kappa_{\mathcal{D},opt}^{(2)}$ to build conditions.  
  We assume that we have all the $2$-torsion points in the Kummer surface $\kappa_{\mathcal{D},opt}^{(2)}$, while we have on the $\kappa_{\mathcal{D}}$ representation
  $\phi(a_1)$, \ldots, $\phi(a_6)$, $\phi(a_{12})$, $\phi(a_{34})$ and $\phi(a_{35})$.
  
We send $0$ to $0$, giving us the conditions $m_4=m_8 =m_{12}=0$. We can fix $m_{16}=1$.
Then, with three nested \emph{for} loops, we test all the $2$-torsion points in $\kappa_{\mathcal{D},opt}^{(2)}$ onto which the points $\phi(a)\in\kappa_{\mathcal{D}}$ can be sent to, for $a\in\{a_3,a_4,a_5\}$. For each
point, this gives $3$ conditions on the $m_i$ ($3$ and not $4$ because of the projectivity). Moreover, as we want to preserve the group structures of the sets of $2$-torsion points, for a choice of the image
of $\phi(a_3)$, $\phi(a_4)$ and $\phi(a_5)$, this fixes an image for $\phi(a_{34})$ and $\phi(a_{35})$.

Thus we obtain $4+3\times 5=19$ conditions on the $16$ $m_i$ and we compute a
Gröbner basis. Once a solution is found, we can verify if the points 
$\phi(a_1)$, $\phi(a_2)$ and $\phi(a_{12})$ in the Kummer surface $\kappa_{\mathcal{D}}$  are sent to $2$-torsion points in $\kappa_{\mathcal{D},opt}^{(2)}$.
\end{enumerate}
\begin{remm} The asymptotic complexity of Gröbner basis is difficult to establish. But here, all the computations do not depend on $\ell$
  (only on the finite field $K$). In practice, the two solutions work very well.
\end{remm}
\begin{remm}\label{rem:Za3}
The points $\phi(a_i)$ for $i\in\{1,\ldots,6\}$ are known since we needed them to compute the equation of $\mathcal{D}$
(in the optimized version). The points $\phi(a_{12})$, $\phi(a_{34})$ and $\phi(a_{35})$ are obtained looking at the intersections $\{Z_{a_{12}}=0,Z_{a_{34}}=0,Z_{a_{35}}=0\}$, $\{Z_{a_{3}}=0,Z_{a_{12}}=0,Z_{a_{34}}=0\}$, $\{Z_{a_3}=0,Z_{a_{12}}=0,Z_{a_{35}}=0\}$ respectively (see Table \ref{table:confG2}).
All these tropes have already been  computed to find the equation of $\mathcal{D}$ except for $Z_{a_3}$ in Algorithm \ref{alg:eval2points}.
It can be computed either now or during this Algorithm 
with the cost of two evaluations of $\eta_f[2[a_3]-2[0],y]$.
\end{remm}

\paragraph{Genus $3$ case.} We proceed with the same idea as in the second solution of the genus $2$ case to go from $\kappa_{\mathcal{D}}$ to $\kappa_{\mathcal{D},opt}^{(3)}$.
This time we have $64$ unknown variables.
We fix a basis of the $2$-torsion points of $J_{\mathcal{D}}$, seen in $\kappa_{\mathcal{D}}$ (this basis contains $6$ points).
We want to send the points in this basis to the $2$-torsion points seen in $\kappa_{\mathcal{D},opt}^{(3)}$ and preserve the group structures.
Using $4$ nested \emph{for} loops, we obtain enough conditions to compute the transformation ($2^3\times 7+1=57$ conditions with $3$ loops
and $2^4\times 7+1=113$ conditions with $4$ loops. Recall that we are in $\mathbb{P}^7$).
However, in practice, this method takes too much time: many hours in our examples, with fields of size $10^5$.

To improve this method, we propose the following solution.
Recall that the Kummer varieties $\kappa_{\mathcal{D}}$ and $\kappa_{\mathcal{D},opt}^{(3)}$ are described with $1$ quadric and around $34$ quartics if the curves are hyperelliptic.
Computing all of them is costly in practice ($330\times 8$ evaluations of $\eta_f$ functions, see Section \ref{sec:kumm}) so instead
we compute only the quadric in the equations of the Kummer variety ($36\times 8$ evaluations).
We look for a transformation that sends a quadric to the other, as in the first solution of the genus $2$ case, which gives us many conditions
on the $64$ unknown variables and facilitate the Gröbner basis computation. Then we proceed in the same way as above but this time, $3$ nested \emph{for} loops
are enough instead of $4$. In our examples, it took around half an hour to test all the possibilities (but a solution was found in a few minutes).
This is still not satisfactory because this solution is still too slow even for small examples and because we do computations which depend on $\ell$.

\subsubsection{Image of a single point.}
Let $(u(t),v(t))$ be a point on $\mathcal{C}(L)$, where $L=K[t]$. We want the image of $P(t)=(u(t),v(t)) - O$ in the Kummer surface represented by $\kappa_{\mathcal{D},opt}^{(2)}$.

We can not directly compute the image of $P$ by the $\eta_f$ function as $P$ is a pole of these functions.
But this is not the case of its multiples in the Jacobian.
It is well-known that Kummer surfaces are not endowed with a group structure but a pseudo-addition law  can be defined on them.
This means that if we have the points $\pm P_1$, $\pm P_2$, $\pm (P_1+P_2)$ in the Kummer surface, then we can compute $\pm (P_1-P_2)$ on it.

Let $m>1$ be an integer. Compute the image of $mP(t)$, $(m+1)P(t)$ and $(2m+1)P(t)$ on the Kummer surface $\kappa_{\mathcal{D}}$
(compute $(\eta_1(nP(t)):\ldots:\eta_4(nP(t)))$ for $n\in\{m,m+1,2m+1\}$), then use the transformation to deduce
the corresponding points on the Kummer surface represented by $\kappa_{\mathcal{D},opt}^{(2)}$,
do the pseudo-addition  
to deduce the image of $P(t)$ by the isogeny on the Kummer surface  $\kappa_{\mathcal{D},opt}^{(2)}$ and deduce from it a point on the Jacobian of
$\mathcal{D}(L)$. Thus, using $\kappa_{\mathcal{D},opt}^{(2)}$ has the double advantage that we can do pseudo-addition in it and that lifting to the
Jacobian is easy.

This step requires $12$ evaluations of $\eta_f$ functions in the field $L$ ($9$ if $\eta_1=\eta_f[2[a_6]-2[0],y]$).

This idea also works in genus $3$. See \cite{Stoll} for the pseudo-addition.

\subsection{Example for hyperelliptic curves of genus $3$}
As the moduli space of hyperelliptic curves is of dimension $5$ in the $6$-dimensional moduli space of genus $3$ curves, 
if we start from a hyperelliptic curve of genus $3$ and a maximal isotropic subgroup of the $\ell$-torsion,
the corresponding isogenous curve is generically non-hyperelliptic.
The nature of the isogenous curve can be established from the type of the configuration
or the equations describing the Kummer threefold, in particular the presence of a quadric.

We have built examples of isogenous hyperelliptic
curves using \cite[Satz 4.4.2]{PhDWeng}, which states that if the Jacobian of $\mathcal{C}$ has complex multiplication by $\mathcal{O}_K$
with $\Q(i)\subset K$ and is simple, then $\mathcal{C}$ is hyperelliptic.
Curves with these properties are provided in \cite{PhDWeng}.
The fact that an isogeny preserves the field
of complex multiplication gives some probability that the isogenous curve is also  hyperelliptic.
For instance, the curves on $\mathbb{F}_{120049}$
  \[\mathcal{C}:Y^2=X^7 + 118263X^5 + 44441X^3 + 81968X,\]
\[\mathcal{D}:Y^2=X^7 + 87967X^6 + 102801X^5 + 70026X^4 + 30426X^3 + 37313X^2 +  77459X\]
  are $(5,5,5)$-isogenous. 
  The Mumford representation of the generators of the isotropic subgroup are
  \[T_1=\langle  X^3 + 90254X^2 + 103950X + 34646, 63966X^2 + 19029X + 62065\rangle,\]
  \[T_2=\langle X^3 + 29700X^2 + 10920X + 14179, 77142X^2 + 66846X + 84040\rangle,\]
  \[T_3=\langle X^3 + 119858X^2 + 87344X + 82114, 51063X^2 + 95007X + 64731\rangle\]
  and the isogeny is described by the following equations

\[\begin{split}
\mathbf{S}=&  (26590u^{13} + 38875u^{12} + 11144u^{11} + 39196u^{10} + 48794u^9 + 
  80531u^8 + 56286u^7 + 42203u^6 +\\ &  49314u^5 + 34405u^4 + 
        28021u^3 + 82360u^2 + 112863u + 64433)/(u^8 + 107005u^7 + 
        34717u^6 + \\ & 96329u^5 + 81848u^4 + 90494u^3),
\end{split}\]
\[\begin{split}
\mathbf{P}=&   (13588u^{13} + 99739u^{12} + 60510u^{11} + 3267u^{10} + 56188u^9 + 
        27913u^8 + 79606u^7 + 79490u^6 + \\&  39953u^5 + 101739u^4 + 
        118959u^3 + 88791u^2 + 59459u + 44419)/(u^8 + 107005u^7 + 
        34717u^6  \\&      + 96329u^5 + 81848u^4 + 90494u^3),\\
        \end{split}\]
\[\begin{split}
        \mathbf{A}=&    (87680u^{12} + 77147u^{11} + 47767u^{10} + 91104u^9 + 101830u^8 + 
        51358u^7 + 106657u^6 +  1059u^5 +  \\ & 28890u^4 + 72926u^3 + 
        40489u^2 + 20614u + 13587)/(u^7 + 107005u^6 + 34717u^5 + 
        96329u^4 + \\ & 81848u^3 + 90494u^2),\\
        \end{split}\]
\[\begin{split}
\mathbf{R}=&    (12306u^{20} + 37665u^{19} + 84758u^{18} + 83076u^{17} + 51365u^{16} + 
        42432u^{15} + 76312u^{14} + \\ & 63248u^{13} + 97292u^{12} + 25304u^{11}
        + 38304u^{10} + 26932u^9 + 108075u^8 + 40558u^7 + 5431u^6 +\\ &
        22057u^5 + 100345u^4 + 113409u^3 + 73221u^2 + 39576u + 
        78248)/(u^{16} + 107005u^{15} + \\ & 32931u^{14} +  103407u^{13} + 
        67011u^{12} + 105334u^{11} + 109571u^{10} + 59270u^9 + 83877u^8 +\\ &
        34998u^7 + 98548u^6 + 24580u^5),\\\end{split}\]
\[\begin{split}
\mathbf{T}=&    (39012u^{20} + 43063u^{19} + 41666u^{18} + 90531u^{17} + 18614u^{16} + 
        112658u^{15} + 99705u^{14} + \\ & 15123u^{13} + 56542u^{12} + 
        44122u^{11} + 40721u^{10} + 103078u^9 + 29236u^8 + 114961u^7 
        + 99184u^6 + \\ & 32122u^5 + 94412u^4 + 42358u^3 + 4616u^2 + 
        66587u + 86686)/(u^{16} + 107005u^{15} + 32931u^{14} + \\ &
        103407u^{13} + 67011u^{12} + 105334u^{11} + 109571u^{10} + 
        59270u^9 + 83877u^8 + 34998u^7 + \\ & 98548u^6 + 24580u^5),\\\end{split}\]
\[\begin{split}
\mathbf{E}=&    (77510u^{19} + 5507u^{18} + 57109u^{17} + 115038u^{16} + 83721u^{15} + 
        32646u^{14} + 7900u^{13} + 28888u^{12} + \\ & 83235u^{11} + 112193u^{10}
        + 99943u^9 + 38123u^8 + 70050u^7 + 48716u^6 + 15860u^5 + 
        65499u^4 + \\ & 38669u^3 + 35838u^2 + 82517u + 82266)/(u^{15} + 
        107005u^{14} + 32931u^{13} + 103407u^{12} + \\ & 67011u^{11} + 
        105334u^{10} + 109571u^9 + 59270u^8 + 83877u^7 + 34998u^6 +
        98548u^5 + 24580u^4).
\end{split}\]

\section{Using algebraic theta functions}\label{sec:compatibility} 

There are well-known formulas allowing one to reconstruct the equation of hyperelliptic and non-hyperelliptic curves of genus $2$ and $3$
from analytic theta constants. According to Mumford's theory, these formulas are (generically) valid
for any field if we use algebraic theta functions \cite{MumEqAVI,MumEqAVII,MumEqAVIII}.
The $\eta_f[2[a]-2[0],y]$ functions have the same divisors as the algebraic theta functions.
We want to find a constant $c_a$ for each $2$-torsion point $a$ such that the functions
$c_a\cdot \eta_f[2[a]-2[0],y]$ satisfy the same algebraic relations as the analytic theta functions.


In this section, we begin by recalling the theory of algebraic theta function and we apply it in our context.
Then we give the definition and some fundamental properties of the analytic theta functions and finally
we explain how to find the constants $c_a$ in genus $2$ and $3$ so that we can use theta based formulas for computing isogenies.

\subsection{Algebraic theta functions}\label{subsec:alg}

This section is based on \cite[Section 1]{MumEqAVI}. See also \cite{DamienThesis} on this subject.
We give examples to illustrate the theory in our case (see also \cite[Section 4]{CouvEz}).

Let $X$ be an abelian variety of dimension $g$ over an algebraically closed field $K$ of characteristic~$p$.
Let $\mathcal{L}$ be an invertible sheaf on $X$. Denote by $H(\mathcal{L})$ the subgroup of closed points $x\in X$ such that
$t_x^*\mathcal{L}\simeq \mathcal{L}$. Here, $t_x:X\to X$ denotes the translation by $x$.
We have that $\mathcal{L}$ is ample if and only if $H(\mathcal{L})$ is finite and $\Gamma(X,\mathcal{L}^n)\ne \{0\}$ for all $n>0$. 
Moreover, there exists a positive integer $d$, called the degree of $\mathcal{L}$, such that $\dim H^0(X,\mathcal{L}^n)=d\cdot n^g$, for all $n\ge 1$.
Assume from now that $p\nmid d$. Then $d^2$ is the cardinality of $H(\mathcal{L})$.  

\begin{exem} Let $X=J_{\mathcal{C}}$ for a curve $\mathcal{C}$ of genus $g$. Consider $\mathcal{L}=\mathcal{O}_{J_{\mathcal{C}}}(W_{-\theta})$ which is of degree $d=1$.
  Then $\mathcal{L}^n$ is of degree $n^g$ and $H(\mathcal{L}^n)$ if of cardinality $n^{2g}$. It is equal to the set of the $n$-torsion points $J_{\mathcal{C}}[n]$.
\end{exem}

Define $\mathcal{G}(\mathcal{L})$ as the set of pairs $(x,\phi_x)$ where $x$ is a closed point of $X$ and $\phi_x$ an isomorphism $\mathcal{L}\to t_x^*\mathcal{L}$.
This is a group for the group law $(y,\phi_y)\cdot (x,\phi_x)=(x+y,t_x^*\phi_y\circ\phi_x)$. The inverse of $(x,\phi_x)$ is clearly $(-x,(t_{-x}^*\phi_x)^{-1})$ and the neutral element is $(0,\mathrm{id})$.
This group is called \emph{the Mumford Theta group}.
The forgetful map $(x,\phi_x)\in \mathcal{G}(\mathcal{L})\mapsto x\in H(\mathcal{L})$ is surjective and the following sequence is exact:
\[ 0\to K^* \to \mathcal{G}(\mathcal{L})\to H(\mathcal{L})\to 0,\]
where $m\in K^*\mapsto (0,[m]) \in \mathcal{G}(\mathcal{L})$ and $[m]$ is the multiplication-by-$m$ automorphism of $\mathcal{L}$.

Let $x,y\in H(\mathcal{L})$ and $\tilde{x}$, $\tilde{y} \in \mathcal{G}(\mathcal{L})$ which lie over $x$ and $y$.
Define $e^{\mathcal{L}}(x,y)=\tilde{x}\cdot \tilde{y}\cdot \tilde{x}^{-1}\cdot \tilde{y}^{-1}$.
This is a non-degenerate skew-symmetric bilinear pairing from $H(\mathcal{L})$ to $K^*$, called the \emph{commutator pairing}.
A \emph{level subgroup} $\tilde{K}$ of  $G(\mathcal{L})$ is a subgroup such that $K^*\cap \tilde{K}=\{0\}$, i.e. $\tilde{K}$ is isomorphic to its image in $H(\mathcal{L})$.
A level subgroup over $K<H(\mathcal{L})$ exists  if and only if $e^{\mathcal{L}}(x,y)=1$ for all $x,y\in K$.

\begin{exem} 
  Let   $\mathcal{L}=\mathcal{O}_{J_{\mathcal{C}}}(\ell W_{-\theta})$ be an invertible sheaf, where $\ell$ is a prime number $\ne p$.
  As in \cite[Section 4]{CouvEz}, for $u\in J_{\mathcal{C}}[\ell](K)$, we let $\theta_u$ be a function with divisor $\ell W_{-\theta+u}-\ell W_{-\theta}$ (i.e. $\theta_u=\eta[\ell[u]-\ell[0]]$).
  Let $\phi_u:f\in H^0(J_{\mathcal{C}},\mathcal{L})\mapsto \theta_u\circ t_u\cdot f\in H^0(J_{\mathcal{C}},t_u^*\mathcal{L})$. Then $(u,\phi_u)\in G(\mathcal{L})$.
  For $u,v\in J_{\mathcal{C}}[\ell]$, we have that $e^{\mathcal{L}}(u,v)=\theta_u\cdot \theta_v\circ t_{-u}\cdot (\theta_u\circ t_{-v})^{-1}\cdot \theta_v^{-1}$, which does not depend
  on the choice of $\theta_u$ and $\theta_v$, which are defined up to a constant.
  Note that the functions $\theta_u\cdot \theta_v\circ t_{-u}$ and $\theta_v\cdot \theta_u\circ t_{-v}$ have the same divisor, so the image of the pairing is indeed in
  (the image in $\mathcal{G}(\mathcal{L})$ of) $K^*$.
  Moreover, they are equal when $e^{\mathcal{L}}(u,v)=1$, and thus when $(u,\phi_u)$ and $(v,\phi_v)$ are in a same level subgroup. In this case, we also have, looking
  at the group law (which is commutative in a level subgroup), that  $\theta_{u+v}=\theta_u\cdot \theta_v\circ t_{-u}$ ($=\theta_v\cdot \theta_u\circ t_{-v}$).
\end{exem}

We can write $H(\mathcal{L})=K_1(\mathcal{L})\oplus K_2(\mathcal{L})$ for some subgroups $K_1(\mathcal{L})$, $K_2(\mathcal{L})$ of $H(\mathcal{L})$ such that  
$e^{\mathcal{L}}(x,y)=1$ for $x,y\in K_1(\mathcal{L})$ or $x,y\in K_2(\mathcal{L})$. Moreover, there is an isomorphism $x\in K_2(\mathcal{L})\mapsto e^{\mathcal{L}}(\cdot,x)\in \mathrm{Hom}(K_1(\mathcal{L}),K^*)$.
Let $\delta=(d_1,\ldots,d_k)$ be the sequence of elementary divisors of $K_1(\mathcal{L})$. We have $d_{i+1}|d_i$ and $d_i>1$. 
The elementary divisors of $H(\mathcal{L})$ are $(d_1,d_1,d_2,d_2,\ldots,d_k,d_k)$. We say that $\mathcal{L}$ is of type $\delta$.

Let $\delta=(d_1,\ldots, d_k)$ be a sequence of positive integers as above. Denote $K(\delta)=\oplus_{i=1}^k\mathbb{Z}/ d_i \Z$, $\widehat{K(\delta)}=\mathrm{Hom}(K(\delta),K^*)$ 
and $H(\delta)=K(\delta)\oplus \widehat{K(\delta)}$. Let $\mathcal{G}(\delta)$ be, as a set, equal to $K^*\times K(\delta)\times \widehat{K(\delta)}$.
This is a group, called the \emph{Heisenberg group}, with the group law $(\alpha,x,l)\cdot (\beta,y,l')=(\alpha\cdot \beta \cdot l'(x),x+y,l+l')$.

If $\mathcal{L}$ is of type $\delta$, then the sequence $0\to K^*\to \mathcal{G}(\mathcal{L})\to H(\mathcal{L})\to 0$ is isomorphic to the sequence
$0\to K^*\to \mathcal{G}(\delta)\to H(\delta)\to 0$.
An isomorphism of $\mathcal{G}(\mathcal{L})$ and $\mathcal{G}(\delta)$ which is the
identity on $K^*$ is called a \emph{$\theta$-structure} on $(X,\mathcal{L})$.

\begin{exem} Still for $\mathcal{L}=\mathcal{O}_{J_{\mathcal{C}}}(\ell W_{-\theta})$. 
  The elementary divisors are $(\ell,\ldots,\ell)$ ($g$ times). Let $\delta=(\ell,\ldots,\ell)$,  $\sigma_1$ an isomorphism
  from $K(\delta) \simeq (\Z/\ell\Z)^g$ to $K_1(\mathcal{L})$ and
  $\sigma_2$ the isomorphism from $\widehat{K(\delta)}$ to $K_2(\mathcal{L})$ determined by $\sigma_1$ and the isomorphism $K_2(\mathcal{L})\simeq \mathrm{Hom}(K_1(\mathcal{L}),K^*)$.
  Let $\tilde{K}_i$ be level subgroups over $K_i(\mathcal{L})$, for $i\in\{1,2\}$.
  Note that $(\alpha,x,l)\in G(\delta)=(\alpha,0,0)\cdot (1,0,l)\cdot (1,x,0)$.
  We send $(\alpha,0,0)$ to the image of $\alpha$ in $G(\mathcal{L})$, $(1,0,x)$ to the point in $\tilde{K}_1$ over $\sigma_1(x)$
  and $(1,0,l)$ to the point in $\tilde{K}_2$ over $\sigma_2(l)$. This determines an isomorphism from $G(\delta)$ to $G(\mathcal{L})$. 
  Remark that $(1,0,l)\cdot (1,x,0)=(l(x)^{-1},0,0)\cdot (1,x,0)\cdot (1,0,l)$. The isomorphism gives the same image in both side because
  $l(x)=e^{\mathcal{L}}(\sigma_1(x),\sigma_2(l))$.  
\end{exem}

Define $U_{(x,\phi_x)}:\Gamma(X,\mathcal{L})\to \Gamma(X,\mathcal{L})$, for $(x,\phi_x)\in\mathcal{G}(\mathcal{L})$, by $U_{(x,\phi_x)}(f)=t_{-x}^*(\phi_x(f))$ for
all $f\in\Gamma(X,\mathcal{L})$. This is an action of the group $\mathcal{G}(\mathcal{L})$ since $U_{(y,\phi_y)}(U_{(x,\phi_x)}(f))=U_{(x+y,t_x^*\phi_y\circ \phi_x)}(f)$.
Let $V(\delta)$ be the vector space of functions $f$ on $K(\delta)$ with values in $K$. Then $\mathcal{G}(\delta)$ acts on $V(\delta)$ by
$U_{(\alpha,x,l)}(f)=\alpha\cdot l\cdot f\circ t_x$.

Assume now that the invertible sheaf $\mathcal{L}$ is very ample.
A $\theta$-structure $\Theta$ determines in a canonical way \emph{one} projective embedding of $X$.
Indeed, there is a unique, up to scalar multiples, isomorphism $\psi: \Gamma(X,\mathcal{L})\to V(\delta)$  which commutes with the action of $\mathcal{G}(\mathcal{L})$
and $\mathcal{G}(\delta)$. This isomorphism induces a unique isomorphism of projectives spaces $\mathbb{P}[\Gamma(X,\mathcal{L})]\to\mathbb{P}[V(\delta)]$.
Then, since $\mathcal{L}$ is very ample, there is a canonical embedding $X\to \mathbb{P}[\Gamma(X,\mathcal{L})]$.
Finally, order the elements $a_1, \ldots, a_m$ of the finite group $K(\delta)$. Then a basis of $V(\delta)$ is composed of the
set of Kronecker delta functions $\delta_i$ at the $a_i$. This defines an isomorphism $\mathbb{P}[V(\delta)]\to \mathbb{P}^{m-1}$. Note that $m$ is equal to the degree $d$ of $\mathcal{L}$.

So the \emph{$\theta$-structure} determines a canonical basis of $\Gamma(X,\mathcal{L})$ up to scalar multiples, and thus an embedding to $\mathbb{P}^{d-1}$.
This basis is $\{\psi(\delta_1),\ldots, \psi(\delta_d)\}$.   
We call these functions the \emph{canonical algebraic theta functions}.

\begin{exem} We continue with the previous examples. The isomorphism $\psi$ satisfies, for all $f\in V(\delta)$ and $(\alpha,x,l)\in G(\delta)$,
  $\psi(U_{(\alpha,x,l)}(f))=U_{\Theta((\alpha,x,l))}(\psi(f))$, that is
  \[\psi(\alpha\cdot l\cdot f\circ t_x)=\alpha\cdot \theta_{\sigma_2(l)}\cdot \theta_{\sigma_1(x)}\circ t_{-\sigma_2(l)}\cdot \psi(f)\circ t_{-\sigma_1(x)-\sigma_2(l)}.\]
  Denote now by $\delta_x$ the Kronecker delta functions at $x\in K(\delta)$.
  We deduce from $\delta_y\circ t_x=\delta_{y-x}$  the equality $\psi(\delta_{y-x})=\theta_x\circ \psi(\delta_y)\circ t_{-x}$, for any $x,y\in K(\delta)$.
  If we take $\psi(\delta_x)$ such that $(\psi(\delta_x))=(\theta_{-x})$, we can see that these equalities are satisfied.  
\end{exem}

The canonical basis of theta functions satisfy many algebraic relations.
For example, the \emph{duplication formula} \cite[Corollaire 4.3.7]{DamienThesis} which is a consequence of the \emph{addition formula} \cite[Théorème 4.3.5]{DamienThesis}
(proved analytically by Koizumi \cite{Koizumi} and then algebraically by Kempf \cite{Kempf89}).
The most important relations are the \emph{Riemann relations} \cite[Théorème 4.4.6]{DamienThesis}.
We give in the next section these equations in the case $K=\mathbb{C}$ only to simplify the exposition.

Following the examples for $\ell=2$, we have that a subset of the $\ell^{2g}$ eta functions $c_a\cdot \eta[2[a]-2[0],y]$, for $a\in J_{\mathcal{C}}[2]$,
form a canonical theta basis, where $c_a$ is a constant and  $y\in J_{\mathcal{C}}$ is fixed.
We do not try to find this basis and the constants $c_a$ using the theory. We do it using the algebraic relations.
As in \cite{CossetRobert,LubRob2}, we will focus on the case $K=\mathbb{C}$ but our formulae and algorithms apply to any field of characteristic $\ne 2$.


\begin{remm} The image in $\mathbb{P}^{d-1}$ of the neutral element of $X$ by the embedding is called the \emph{theta-null point}.
  Conversely, at some conditions, the theta-null point determines  $X$, $\mathcal{L}$  and a $\theta$-structure (see \cite[Page 94]{DamienThesis}).
  Furthermore, if $f:(X,\mathcal{L}_X,\Theta_{\mathcal{L}_X}) \to (Y,\mathcal{L}_Y,\Theta_{\mathcal{L}_Y})$ is an isogeny of polarized abelian varieties
  with theta structure, there is a theorem called \emph{the isogeny theorem} (see for example \cite[Proposition 2.2]{LubRob}) relating the canonical bases induced by the $\theta$-structures.
  The method of \cite{CossetRobert,LubRob,LubRob2} is based on this theorem and a precise use of the theory of algebraic theta functions for any abelian variety (and not only Jacobians of curves).
\end{remm}

\subsection{Analytic theta functions}
Analytic theta functions have been widely studied and are well understood from many points of views.
Good references are \cite{Mumford83,Mumford84, Birk}.
In this section, $g$ is an integer $\ge 1$.

Let $z\in\C^g$ and $\Omega$ in the Siegel upper-half space $\HH_g$ (the $g\times g$ symmetric matrices over the complex numbers with positive definite
imaginary part). The classical theta function is 
\[ \theta(z,\Omega)=\sum_{n\in\Z^g}\exp{(i\pi\,^t\!n\Omega n+2 i \pi\,^t\!n z)}\]
and the classical theta function with characteristic $(m,n)$, where $m,n\in\Q^g$, is
\begin{equation}\label{eq:thetacar}\thetacar{m}{n}(z,\Omega)=\exp{(i\pi\,^t\!m\Omega m+2i\pi\,^t\!m(z+n))}\cdot\theta(z+\Omega m+n,\Omega).\end{equation}
Let $r$ be an integer $\ge 2$ and $\Omega$ fixed. Then the $r^{2g}$ theta functions of the form $\thetacar{m}{n}(z,\Omega)^r$
for  $m,n$ representatives of the classes of $\frac 1r\Z^g/\Z^g$ are said to be of \emph{level $r$}.
Then $r^g$ linearly independent functions among them provide an embedding of the abelian variety seen as the torus $\C^g/(\Omega\Z^{g}+\Z^{g})$ to $\mathbb{P}^{r^g-1}(\C)$
unless $r=2$ where the embedding is only from the Kummer variety $\C^g/(\Omega\Z^{g}+\Z^{g})/\sim$, for $\sim$ the equivalence relation such that $z\sim -z$.
Many bases and relations between them can be found in \cite[Chapitre 3]{Cosset}.

Let $m,n\in\Q^g$ and  $m_1,m_2\in\Z^g$. According to \cite[Page 123]{Mumford83} we have 
\[
  \thetacar{m}{n}(z+\Omega m_1+m_2,\Omega)=\exp{(-i\pi\,^t\!m_1\Omega m_1-2 i\pi\,^t\!m_1 z)}\cdot\exp{(2i\pi(\,^t\!mm_2-\,^t\!nm_1))}\cdot\thetacar{m}{n}(z,\Omega),\]
\begin{equation}\label{eq:eq1} 
  \mathrm{and}\qquad\thetacar{m+m_1}{n+m_2}(z,\Omega)=\exp{(2i\pi\,^t\!mm_2)}\cdot\thetacar{m}{n}(z,\Omega).\end{equation}
Moreover, using the definitions, we have the equality $\thetacar{m}{n}(-z,\Omega)=\thetacar{-m}{-n}(z,\Omega)$.

Let $\mathcal{C}$ be a smooth projective curve of genus $g$ over $\C$ and $W$ be the image of the symmetric product $\mathcal{C}^{(g-1)}$ in $\Pic^{g-1}(\mathcal{C})$
(as in Section \ref{subsec:def}).
Denote by $\Theta$ the zero divisor of $\thetacar{0}{0}(z,\Omega)$ ($\Omega$ fixed corresponding to $\mathcal{C}$).
According to \cite[Chapter II, Theorem 3.10]{Mumford83}, there exists a theta characteristic $\theta$
($\theta$ is a linear equivalence class of divisors of degree $g-1$ and $2\theta$ is the canonical class)
such that the image by the Abel-Jacobi map of $W_{-\theta}$ is $\Theta$.
From now on, let $m$, $n$ be representatives of $\frac 12\Z^{g}/\Z^{g}$. Using Equations (\ref{eq:thetacar}) and (\ref{eq:eq1}), we have 
\[\thetacar{m}{n}(z+\Omega m+n,\Omega)=\exp{(-i\pi\,^t\!m \Omega m-2i\pi \,^t\!m z+4i\pi\,^t\!m n)}\cdot\thetacar{0}{0}(z,\Omega)\]
from which we deduce 
\[\thetacar{m}{n}(z,\Omega)=\thetacar{m}{n}((z-\Omega m-n)+(\Omega m+n),\Omega)\]
\[=\exp{(i\pi\,^t\!m \Omega m-2i\pi \,^t\!m z+6i\pi\,^t\!m n)}\cdot\thetacar{0}{0}(z-\Omega m-n,\Omega).\]
The divisor of $\thetacar{m}{n}(z,\Omega)$ ($\Omega$ is fixed) is then $\Theta_{\Omega m+n}$.
Note that $\Omega m+n$ is a $2$-torsion point. We are interested in the functions
of the form $\thetacar{m}{n}(z,\Omega)^2/\thetacar{0}{0}(z,\Omega)^2$ having divisor $2\Theta_{\Omega m+n}-2\Theta$
which is similar to the ones of the  $2^{2g}$ ($g\in\{2,3\}$) level $2$
 functions $\eta_f[2[a]-2[0],y]$ (for the $2$-torsion point $a$ in $J_{\mathcal{C}}$).

 Define  \[\theta_{m,n}(z):=c_{m,n}\cdot\thetacar{m}{n}(z,\Omega)^2/\thetacar{0}{0}(z,\Omega)^2,\]
for some constants $c_{m,n}$, $m$ and $n$ representatives of the classes of $\frac 12\Z^g/\Z^g$ and some $\Omega$ fixed (corresponding to $\mathcal{C}$).
We want to multiply the $\eta_f[2[a]-2[0],y]$ functions by constants such that these new functions verify the same algebraic relations as the analytic theta functions. We speak then of algebraic theta functions.
Applying the previous equalities, we obtain
\[\theta_{m,n}(\Omega m+n)=c_{m,n}\cdot\exp{(-2i\pi\,^t\!m\Omega m)}\cdot\thetacar{0}{0}(0,\Omega)^2/\thetacar{0}{0}(\Omega m +n,\Omega)^2\]
and \[\theta_{m,n}(0)=c_{m,n}\cdot\exp{(2i\pi\,^t\!m\Omega m+4i\pi\,^t\!m n)}\cdot\thetacar{0}{0}(\Omega m +n,\Omega)^2/\thetacar{0}{0}(0,\Omega)^2\]
if the denominator is not $0$.
Finally, the product of these two functions gives us the following fundamental relation
\begin{equation}\label{eq:gtorg0} \theta_{m,n}(\Omega m+n)\theta_{m,n}(0)=c_{m,n}^2\exp{(4i\pi\,^t\!m n)}.\end{equation}  
We will use this relation with $\eta_f[2[a]-2[0]]$ to deduce a constant $c_a$ corresponding to $c_{m,n}^2$.
We will then explain how to choose a square root of $c_a$ (for all $a$) and explain that multiple choices are possibles (See Remark \ref{rem:choix}). 

A lot of algebraic relations between the analytic theta functions can be deduced from the two following propositions.
  \begin{prop}[Riemann's theta formula] Let $m_1$,$m_2$,$m_3$,$m_4$ in $\mathbb{R}^{2g}$. Put
    $n_1=\frac 12 (m_1+m_2+m_3+m_4)$, $n_2=\frac 12 (m_1+m_2-m_3-m_4)$, $n_3=\frac 12 (m_1-m_2+m_3-m_4)$, $n_4=\frac 12 (m_1-m_2-m_3+m_4)$.
    Then
    \[\theta_{m_1}\theta_{m_2}\theta_{m_3}\theta_{m_4}=\frac 1{2^g}\sum_{\alpha}\exp{(4i\pi m'_1\,^t\!\alpha'')}\theta_{n_1+\alpha}\theta_{n_2+\alpha}\theta_{n_3+\alpha}\theta_{n_4+\alpha},\]
    where, for $m\in\mathbb{R}^{2g}$, we denote $m=(m',m'')$ and $\theta_m=\thetacar{m'}{m''}(0,\Omega)$ and where $\alpha$ runs over a complete set
    of representatives of $\frac 12\Z^{2g}/\Z^{2g}$.
\end{prop}
\begin{proof} See \cite[Chapter IV, Theorem 1]{Igu72}.\end{proof}
\begin{prop}[Duplication formula] For  $m$, $n$  representatives of $\frac 12\Z^{g}/\Z^{g}$, 
  \[\thetacar{m}{n}(z,\Omega)^2=\frac 1{2^g}\sum_{\beta\in\frac12\Z^g/\Z^g}\exp{(4i\pi\,^t\! m\beta)}\thetacar{0}{n+\beta}(z,\frac \Omega 2)\thetacar{0}{n}(z,\frac \Omega 2).\]
  \end{prop}
\begin{proof} See \cite[Chapter IV, Theorem 2]{Igu72}.\end{proof}

\subsection{Genus $2$ case: Rosenhain invariants} \label{subsec:Rosenhain}


A genus $2$ curve can be written in the Rosenhain form $Y^2=X(X-1)(X-\frak{r}_1)(X-\frak{r}_2)(X-\frak{r}_3)$, where,
over $\C$, we have 
\[\frak{r}_1=\frac{\theta_0^2\theta_1^2}{\theta_3^2\theta_2^2},\qquad \frak{r}_2=\frac{\theta_1^2\theta_{12}^2}{\theta_2^2\theta_{15}^2},\qquad
\frak{r}_3=\frac{\theta_0^2\theta_{12}^2}{\theta_3^2\theta_{15}^2}.\]
Here, we denote the analytic theta constants (the theta functions for $z=0$) of level $2$ using Dupont's notation
\[\theta_{n_0+2n_1+4m_0+8m_1}(\Omega):=\thetacar{m/2}{n/2}(0,\Omega)\]
for $m=\,^t\!(m_0,m_1)$, $n=\,^t\!(n_0,n_1)$ and $m_i,n_i\in\{0,1\}^2$. We drop the $\Omega$ when we work on a fixed abelian variety.

There are $16$ theta constants and $6$ among them are identically zero: the odd theta constants, that is, those for which
$\,^t\!mn\equiv 1 \bmod 2$. Otherwise we speak of even theta constants. 

We come back to the algebraic case with the notations of Section \ref{subsec:isoG2}.
Let $e_1,e_2,f_1,f_2$ be a symplectic basis of the $2$-torsion of some  genus $2$ curve with an imaginary model.
We want to find the unique $2$-torsion point which is at the intersection of the $6$ tropes of the form $Z_a$ for $a$ a  $2$-torsion point having odd
characteristic (with respect to the fixed symplectic basis).

According
to Proposition \ref{prop:confG2}, if we put
$a''=\begin{ppsmallmatrix}1&1\\1&1\end{ppsmallmatrix}$, then the image in $\mathbb{P}^3$ of any of the six $2$-torsion points having odd characteristic
lie in the trope $Z_{a_0+a''}$, where $a_0$ is defined in this proposition.
The trope $Z_{a_6}$ contains the image of the points $\{a_1,\ldots,a_6\}$; thus $Z_{a_0+a''}$ contains
the image of $\{a_1+a_0+a'',\ldots,a_6+a_0+a''\}$ and 
\[\{Z_{a_1+a_0+a''}=0,\ldots,Z_{a_6+a_0+a''}=0\}=\{a_0+a''\}.\] 
The $2$-torsion point $a_0+a''$ is the one corresponding to $z=0$ (with
respect to the chosen symplectic basis).

  \begin{remm}\label{rem:calcula0}
This gives another way of computing $a_0$: compute the unique point at the inter-
section of the six tropes $Z_a$ with $a$ of odd characteristic and add $\begin{ppsmallmatrix}1&1\\1&1\end{ppsmallmatrix}$
to the result.
    \end{remm}

  We propose the following algorithm to compute the equation of the isogenous curve $\mathcal{D}$ using the algebraic theta functions
  and the Rosenhain form. We do not try to minimize the number of evaluations of $\eta_f$ functions.

\begin{enumerate}
\item Compute all the tropes $Z_a$ for a fixed basis of level $2$ functions where the first element of the basis is $\eta_f[2[a_6]-2[0],y]$.
\item Deduce from it the image of the $2$-torsion points in $\mathbb{P}^3$.
\item For all $a\in J_{\mathcal{C}}[2]$ and $a\not\in\{a_1,\ldots,a_6\}$, we take a lift $a'$ in $\mathbb{A}^4$ of its image in $\mathbb{P}^3$, evaluate
all the tropes at $a'$ and divide by the value obtained in evaluating $Z_{a_6}$ at $a'$ (because $\eta_f[2[a_6]-2[0],y](a)=1$)
so that we obtain $(\eta_f[2[a_1]-2[0],y](a),\ldots,\eta_f[2[a_{45}]-2[0],y](a))\in\mathbb{A}^{16}$.  

\item Let $a''=\begin{ppsmallmatrix}1&1\\1&1\end{ppsmallmatrix}$.
  Following Equation (\ref{eq:gtorg0}), compute $\eta_f[2[a]-2[0],y](a_0+a'')\cdot\eta_f[2[a]-2[0],y](a_0+a''+a)$, for $a\in J_\mathcal{C}[2]$ with even characteristic.
  This gives us $c_a\ne 0$ and we have the algebraic counterpart of $\theta_i^4/\theta_0^4$ ($\ne 0$ for $i\in\{0,1,2,3,4,6,8,9,12,15\}$).
  \begin{remm} 
    Note that $a_0+a''\not\in\{a_1,\ldots,a_6\}$ because otherwise the point $\begin{ppsmallmatrix}0&0\\0&0\end{ppsmallmatrix}$ would be in $Z_{a_0 +a''}$ but it is of even characteristic. Moreover,  $a_0+a''+a\not\in\{a_1,\ldots,a_6\}$ if $a$ is even because otherwise $a$ is in  $\{a_0+a''+a_i\}_{i\in\{1,\ldots,6\}}$ which are exactly the $2$-torsion points having odd characteristic.
    \end{remm}
\item For the Rosenhain invariants, we need (the algebraic counterpart of) $\theta_i^2/\theta_0^2$ for $i\in\{1,2,3,12,15\}$ which we know,
  taking square roots, up to a sign.
More precisely, we need $\frac{\theta_0^2}{\theta_3^2}$, $\frac{\theta_1^2}{\theta_2^2}$ and $\frac{\theta_{12}^2}{\theta_{15}^2}$.
\begin{remm} We can obtain $8$ curves because there are $2^3$ possibilities of sign giving us $8$ triples of Rosenhain invariants.
  One of them or its twist is isogenous to  $\mathcal{C}$.
  Assuming we know the cardinality of the Jacobian $J_{\mathcal{C}}$, we can find the isogenous curve in computing and comparing the cardinalities.
  \end{remm}
But using the algebraic relations between the theta constants we can directly determine the good curve.
\begin{enumerate}
\item We have that
  $(\theta_4\theta_6)^2=(\theta_0\theta_2)^2-(\theta_1\theta_3)^2$.
    \begin{remm}
  This property can be proven using the Duplication formula to write all the $\theta_i^2(\Omega)$ using the $\theta_j(\Omega/2)$ for $j\in\{0,1,2,3\}$
  and comparing the two sides of the equality.
  \end{remm}
Squaring, we obtain $(\theta_4\theta_6)^4=(\theta_0\theta_2)^4+(\theta_1\theta_3)^4-2(\theta_0\theta_2\theta_1\theta_3)^2$.
  Then \[\left(\frac{\theta_4\theta_6}{\theta_0^2}\right)^4-\left(\frac{\theta_2}{\theta_0}\right)^4-
  \left(\frac{\theta_1\theta_3}{\theta_0^2}\right)^4=-2\left(\frac{\theta_2\theta_1\theta_3}{\theta_0^3}\right)^2.\]
  We already know the algebraic counterpart of the LHS and of the square of the RHS.
  We deduce from this equality the good choice of square root of the algebraic counterpart of
$\left(\frac{\theta_2\theta_1\theta_3}{\theta_0^3}\right)^4$ which is the same as for the algebraic counterpart of $\frak{r}_1=\left(\frac{\theta_0\theta_1}{\theta_3\theta_2}\right)^2$.
\item Similarly for $\frak{r}_2$ using $(\theta_4\theta_9)^2=(\theta_1\theta_{12})^2-(\theta_2\theta_{15})^2$.
\item Use $\frak{r}_1\frak{r}_2\frak{r}_3=\frac{\theta_0^4\theta_1^4\theta_{12}^4}{\theta_2^4\theta_3^4\theta_{15}^4}$ to deduce 
the value of the algebraic counterpart of $\frak{r}_3$.
  \end{enumerate}
  \end{enumerate}

\begin{remm}
The fact that the Rosenhain invariants can be determined with the knowledge of quotients of fourth power of theta constants is not surprising
as both are generators for the modular functions invariant by $\Gamma_2(2)$.
\end{remm}
\begin{remm}\label{rem:choix}
The functions $\theta_i^2/\theta_0^2$ are invariant
for $\Gamma_2(2,4)$ and the index $[\Gamma_2(2):\Gamma_2(2,4)]$ is $16$
so that the choice of the square roots we have to take is determined by the choice of $4$ well-choosen quotients (forming a basis).
If we need the algebraic counterpart of the  $\theta_i^2/\theta_0^2$, we generate many
relations from the Duplication formula as we have done before
and do a Gröbner basis for determining relations between the unknown signs. Finally, we can take a
random choice of square roots for the $4$ determining the system because each choice correspond to the evaluation
of the theta constants at $\gamma\cdot\Omega$ for $\gamma\in\Gamma_2(2)/\Gamma_2(2,4)$ 
(and this does not change the isomorphism class of the underlying genus $2$ curve).
\end{remm}

\subsection{Non-hyperelliptic curves of genus $3$}\label{subsec:bitangent}  

We now focus  on the case of non-hyperelliptic curves $\mathcal{D}$ of genus $3$ over a field $K$.
Assume $K$ is algebraically closed. 
We have seen in Proposition \ref{prop:confG3} that the Kummer variety of such a curve has a $(64,28)$-configuration and we can apply similar techniques as in the hyperelliptic case
to compute the tropes and the image of the $2$-torsion points in $\mathbb{P}^7$.
However, we do not know if there is a parameterization allowing one to recover the equation of the curve with these data.
The only way we have found consists in using theta based formulas and the theory of bitangents (see \cite{Weber1876,Rie98}).
The following exposition is based on \cite{TheseChristophe,Ritzenthaler2004} and we refer to these references for more details.

As the curve $\mathcal{D}/K$ is non-hyperelliptic, it can be  embedded as a non-singular plane quartic in $\mathbb{P}^2$. We denote by $x_1$, $x_2$, $x_3$
the coordinates in this projective plane.

\begin{deff}\label{def:bit} A line $l$ is called a bitangent of $\mathcal{D}$ if the intersection divisor $(l\cdot\mathcal{D})$ is of the form $2P+2Q$ for some points $P,Q$ of $\mathcal{D}$.
  If $P=Q$, the point $P$ is called a hyperflex.
\end{deff}

Let $\mathcal{K}$ be the canonical bundle and let $\Sigma=\{L\in \Pic^2(\mathcal{D}):L^2=\mathcal{K}\}$ be the set of theta characteristic bundles.
This set is composed of the two disjoint subsets  $\Sigma_i=\{L\in\Sigma:h^0(L)=i\}$ of even ($i=0$) and odd ($i=1$) theta bundles.
There is a canonical bijection between the set of bitangents, $\Sigma_1$ and the set of odd characteristics for a fixed symplectic basis of the
$2$-torsion. We can deduce from it the following proposition.
\begin{prop} A smooth plane quartic has exactly $28$ bitangents. \end{prop}
The $(64,28)$-configuration comes from this proposition. Indeed, if $l$ is a bitangent and $(l\cdot\mathcal{D})=2P+2Q$, then $2P+2Q$ is a canonical divisor
and $P+Q$ is a theta characteristic as in Section \ref{subsec:def} (from which we can build the $\eta$ and $\eta_f$ functions).
Then if $l'$ is another bitangent giving us the points $P'$ and $Q'$, then the divisor $P'+Q'-P-Q$ is in $W_{-P-Q}$ and it is a $2$-torsion point.
Only the $28$ $2$-torsion points coming from bitangents are in $W_{-P-Q}$.

The equation of $\mathcal{D}$ as a plane quartic is determined and can be reconstructed
knowing the equations of $7$ bitangents forming an Aronhold system (see \cite{CapSern,Lehavi2005}), which is a set of $7$ bitangents such that
if we take $3$ bitangents among these $7$, then the points at which these bitangents intersect the plane quartic do not lie on a conic in $\mathbb{P}^2$.

There exist $288$ Aronhold systems for a given plane quartic and we focus on the following one.

\begin{prop}\label{prop:bitangente}
  An Aronhold system of bitangents for a quartic  is 
  $\beta_1:x_1=0,\quad \beta_2:x_2=0,\quad\beta_3:x_3=0,\quad \beta_4:x_1+x_2+x_3=0,\quad\beta_5:\alpha_{11}x_1+\alpha_{12}x_2+\alpha_{13}x_3=0,\quad
  \beta_6:\alpha_{21}x_1+\alpha_{22}x_2+\alpha_{23}x_3=0,\quad \beta_7:\alpha_{31}x_1+\alpha_{32}x_2+\alpha_{33}x_3=0$, for $[\alpha_{i1}:\alpha_{i2}:\alpha_{i3}]\in\mathbb{P}^2$.    
\end{prop} 

In our case, we do not have the embedding to $\mathbb{P}^2$ because it seems to us that we can not construct it with $\eta_f$ functions
(what would the zero-cycle $\frak{u}$ be ?). However, when $K=\mathbb{C}$, we can find in \cite{Alessio} the following expression of $\alpha_{ij}$ with analytic theta constants.
We fix a symplectic basis and use the following notation
\[\theta_{n_0+2n_1+4n_2+8m_0+16m_1+32m_2}(\Omega):=\thetacar{m/2}{n/2}(0,\Omega)\]
for $m=\,^t\!(m_0,m_1,m_2)$, $n=\,^t\!(n_0,n_1,n_2)$ and $m_i,n_i\in\{0,1\}^2$.  Then
\[
    \alpha_{11}=\frac{\theta_{12}\theta_{5}}{\theta_{33}\theta_{40}},\quad \alpha_{21}=\frac{\theta_{27}\theta_{5}}{\theta_{54}\theta_{40}},\quad \alpha_{31}=-\frac{\theta_{12}\theta_{27}}{\theta_{33}\theta_{54}},\]\[
    \alpha_{12}=\frac{\theta_{21}\theta_{28}}{\theta_{56}\theta_{49}},\quad \alpha_{22}=\frac{\theta_{2}\theta_{28}}{\theta_{47}\theta_{49}},\quad \alpha_{32}=\frac{\theta_{2}\theta_{21}}{\theta_{47}\theta_{56}},\]\[
    \alpha_{13}=\frac{\theta_{7}\theta_{14}}{\theta_{42}\theta_{35}},\quad \alpha_{23}=\frac{\theta_{16}\theta_{14}}{\theta_{61}\theta_{35}},\quad \alpha_{33}=\frac{\theta_{16}\theta_{7}}{\theta_{61}\theta_{42}}.
    \]
    
    The reconstruction of the plane quartic from its bitangents comes from the following result.

\begin{theo}[Riemann]
  Let $\beta_1,\ldots,\beta_7$ be an Aronhold system of bitangents as in Proposition~\ref{prop:bitangente}. Then an equation for the curve is
  \[(x_1\xi_1+x_2\xi_2-x_3\xi_3)^2-4x_1\xi_1 x_2\xi_2 =0\]
where $\xi_1$, $\xi_2$, $\xi_3$ are given by
\[\left\{\begin{array}{l}\xi_1+\xi_2+\xi_3+x_1+x_2+x_3=0,\\ \frac{\xi_1}{\alpha_{i1}}+\frac{\xi_2}{\alpha_{i2}}+\frac{\xi_3}{\alpha_{i3}}+k_i(\alpha_{i1}x_1+\alpha_{i2}x_2+\alpha_{i3}x_3)=0,\qquad i\in\{1,2,3\}\end{array}\right.\]
with $k_1$, $k_2$, $k_3$ solutions of 
\[\left (\begin{matrix}\frac{1}{\alpha_{11}}&\frac{1}{\alpha_{21}}&\frac{1}{\alpha_{31}}\\\frac{1}{\alpha_{12}}&\frac{1}{\alpha_{22}}&\frac{1}{\alpha_{32}}\\\frac{1}{\alpha_{13}}&\frac{1}{\alpha_{23}}&\frac{1}{\alpha_{33}}\end{matrix}\right )
  \left (\begin{matrix}\lambda_1\\\lambda_2\\\lambda_3\end{matrix}\right )=\left (\begin{matrix}-1\\-1\\-1\end{matrix}\right ),\qquad 
\left (\begin{matrix} \lambda_1\alpha_{11}& \lambda_2\alpha_{21}& \lambda_3\alpha_{31}\\\lambda_1\alpha_{12}& \lambda_2\alpha_{22}& \lambda_3\alpha_{32}\\\lambda_1\alpha_{13}& \lambda_2\alpha_{23}& \lambda_3\alpha_{33}\end{matrix}\right )
  \left (\begin{matrix}k_1\\k_2\\k_3\end{matrix}\right )=\left (\begin{matrix}-1\\-1\\-1\end{matrix}\right ).\]
\end{theo}

It remains to explain how to compute the values $\alpha_{ij}$. We proceed as in Section~\ref{subsec:Rosenhain}.
Coming back to notation in Section \ref{subsec:def}, choose $O$ a rational point and $\theta$ a theta characteristic.
Then the divisor of Equation (\ref{eq:varthetasym}) is symmetric so that the $\eta$ and $\eta_f$ functions are even and we use Equation (\ref{eq:evaltheta})
to evaluate them.
(If there is a hyperflex point (see Definition \ref{def:bit}) then take for $O$ this point. Then $\theta=2O$ is a theta characteristic and $\vartheta=0$).
\begin{enumerate}  
\item 
Assuming we have all the tropes and the image of the $2$-torsion points in $\mathbb{P}^7$, we deduce the evaluation of the $\eta_f[2[a]-2[0],y]$ at the torsion points $a'$
where $\eta_f[2[0]-2[0],y](a')\ne 0$, as in the third step of the algorithm based on the Rosenhain invariants.
\item Use Equation (\ref{eq:gtorg0}) to 
  multiply  $\eta_f[2[a]-2[0],y]^2$ by a constant $c_a$ so that we obtain the algebraic counterpart of $\theta_i^4/\theta_0^4$.
  We can not choose square roots randomly.
\item  Generate many relations between squares of theta constants using the Duplication formula or the Riemann theta formula.
  Build a system of equations where the $64$ unknows represent the value $-1$ or $1$.
  Do a Gröbner basis to find $6$ unknow ($[\Gamma_3(2):\Gamma_3(2,4)]=2^6$)  determining the system of equations.
  Take a random choice of square roots for these $6$ unknows. Thus we obtain algebraic theta functions $c_a'\eta_f[2[a]-2[0],y]$ for some constants $c_a'$ such that
  $c_a'^2=c_a$ corresponds to the $\theta_i^2/\theta_0^2$.
\item 
Finally, for the projective point $(\alpha_{11}:\alpha_{12}:\alpha_{13})$, choose any square root of the algebraic counterpart of $\alpha_{11}^2$
and then consider the following equalities (coming from Riemann's theta formula)
\[\begin{array}{c}
\theta_{61}\theta_{45}\theta_{16}\theta_{0}-\theta_{56}\theta_{40}\theta_{21}\theta_{5}+\theta_{49}\theta_{33}\theta_{28}\theta_{12}=0,\\
    \theta_{5}\theta_{12}\theta_{33}\theta_{40}-\theta_{21}\theta_{28}\theta_{49}\theta_{56}-\theta_{42}\theta_{35}\theta_{14}\theta_{7}=0.\\
  \end{array}.\]
From the first one, we have \[(\theta_{61}\theta_{45}\theta_{16}\theta_{0})^2=(\theta_{56}\theta_{40}\theta_{21}\theta_{5})^2+(\theta_{49}\theta_{33}\theta_{28}\theta_{12})^2-2\theta_{56}\theta_{40}\theta_{21}\theta_{5}\theta_{49}\theta_{33}\theta_{28}\theta_{12}\] from which we deduce the good square root of $\alpha_{12}^2$. Similarly, the second equality gives us $\alpha_{13}$.
\item For the  other two projective points, we proceed in the same way using
\[\begin{array}{c}
\theta_{49}\theta_{47}\theta_{28}\theta_{2}-\theta_{54}\theta_{40}\theta_{27}\theta_{5}-\theta_{61}\theta_{35}\theta_{16}\theta_{14}=0,\\
    \theta_{54}\theta_{47}\theta_{27}\theta_{2}-\theta_{49}\theta_{40}\theta_{28}\theta_{5}+\theta_{56}\theta_{33}\theta_{21}\theta_{12}=0,\\
    \\
    -\theta_{55}\theta_{32}\theta_{20}\theta_{3}+\theta_{54}\theta_{33}\theta_{21}\theta_{2}+\theta_{56}\theta_{47}\theta_{27}\theta_{12}=0,\\
\theta_{54}\theta_{33}\theta_{27}\theta_{12}-\theta_{56}\theta_{47}\theta_{21}\theta_{2}+\theta_{61}\theta_{42}\theta_{16}\theta_{7}=0.
  \end{array}\]

\end{enumerate}

\section{Implementation}\label{sec:impl}

We have implemented all the algorithms presented here using the computational algebra system
\emph{Magma} \cite{Magma}. In the case of hyperelliptic curves, the reduced divisors are represented by
their Mumford representation and addition between the reduced divisors $x_1$, $x_2$ is done with
Cantor’s algorithm, giving us the reduced divisor of $x_1 + x_2$. In the non-hyperelliptic
case (genus $3$), we have represented divisors as formal sums of points and used the \emph{Reduction}
function of Magma to reduce divisors. In genus $3$, we did not care about efficiency. We just wanted to test
our algorithms.

In this paper, we only have  optimized the number of evaluations of $\eta_f$ functions in the
genus~$2$ case using the parameterization method. The method computing the isogeny directly
in the Rosenhain form with algebraic theta functions requires the computation of more than
$6$ tropes so it is less efficient than the other method.
We did not optimize our implementation in
this case but computed all the tropes (from which we deduce all the $\phi(a)\in\mathbb{P}^3$) and verified that
it worked. For genus $3$, we did not care about optimization. The point-counting algorithms are
not efficient in practice and computing isotropic subgroups is hard. In the non-hyperelliptic
case, we only tested our algorithm using $\eta$ functions instead of $\eta_f$ functions, so without
computing isogenies. This should not matter because what we want to do is being able to
compute the equation of the curve from the geometry of its Kummer threefold in $\mathbb{P}^7$.
Note that it is easy to verify that two curves are isomorphic. It is enough to compute
invariants of isomorphism classes (Igusa invariants in genus $2$, Shioda invariants for hyperelliptic
curves of genus $3$ and Dixmier-Ohno invariants for plane quartics).
We now give  an example of computation we have done for genus $2$ curves using the
parameterization method. Let
\[\mathcal{C}:Y^2=74737X^5 + 28408X^4 + 89322X^3 + 47216X^2 + 55281X + 86566\]
be a genus $2$ curve over $\mathbb{F}_{100019}$ whose Weierstrass points live in $\mathbb{F}_{100019^5}$. Then $\mathcal{C}$ is
$(7,7)$-isogenous to the curve
\[\mathcal{D}:Y^2=34480X^5 + 27167X^4 + 78914X^3 + 49217X^2 + 75306X + 92103.\]
We do not expose the isotropic subgroup as it is too big to be put here. It lives in an extension of
$\mathbb{F}_{100019}$ of degree $30$. The computation of $\mathcal{D}$ took $27$ seconds on
a $2.50$GHz $64$-bit Intel Core i5-7300HQ.
This includes the computation
of the trope $Z_{a_3}$ (see Remark \ref{rem:Za3}) so that we have the image of the points $\{a_1,\ldots,a_ 6,a_{12},a_{34},a_{35}\}$ in $\mathbb{P}^3$ . The
computation of the matrix allowing one to go from a representation of the Kummer surface to the good
one took slightly less than $1$ second. Computing the image of a single formal point at small
precision, which means $9$ evaluations of $\eta_f$ functions (if $\eta_{a_6}$ is in the basis, $12$ otherwise) took
around $30$ seconds. Extending the precision can be done in $0.4$ seconds and the reconstruction
of rational fractions in $0.02$ seconds. At the end, we can verify the correctness of the rational fractions in
testing if the image of a point is in the Jacobian of $\mathcal{D}$ and by testing the homomorphic property
of the isogeny.

Our implementation is not fast compared to the one of AVIsogenies \cite{Avi} at small primes but we beat them at some examples.
To compare correctly we should consider the size of $\ell$, the size of the finite field $K$ and also the degree of the extension
containing  the $2$-torsion points and the degree of the extension containing $\mathcal{V}$.
But the method exposed here is promising compared to the AVIsogenies method because the complexities in the prime number $\ell$
are $\tilde{O}(\ell^g)$ against $\tilde{O}(\ell^{rg/2})$ with $r=2$ if $\ell$ is a sum of two squares and $r=4$ otherwise.
Indeed, in our algorithm,  only the complexity of the evaluation of $\eta_f$ functions at points of $J_{\mathcal{C}}$ depends on $\ell$
and the complexity of an evaluation is established in Theorem \ref{theo:evaletaf}.

An improvement would be to be able to evaluate the $\eta_f$ functions at $2$-torsion points (and at non-generic points) directly.
This would reduce the number of evaluations of $\eta_f$ functions.
We manage to have results when using $\eta$ functions but nothing with $\eta_f$ functions.
In the latter case, our code returned some errors and we were not able to fully understand them from
a theoretical point of view and from a practical one.

\section*{Acknowledgements}
This research has been done while the author was working as a postdoc in the Caramba team in Nancy. I thank
Pierrick Gaudry for suggesting me to work on \cite{CouvEz}, for helping me in my first footsteps and for
giving me the idea of using the pseudo-addition law on the Kummer variety. I also thank Christophe
Ritzenthaler for a fruitful discussion we had, leading me to Equation~(\ref{eq:gtorg0}).
I thank Jean-Marc Couveignes for answering to the many questions I asked him about his paper and Damien Robert for his help.
Finally, I thank the anonymous reviewer whose comments have greatly improved this manuscript.

\bibliography{bibJMT}

The author can be contacted at enea.milio@gmail.com

\end{document}